\documentclass[12pt]{amsart}
\usepackage{graphicx}
\usepackage[headings]{fullpage}
\usepackage{bm,amssymb,epic,eepic,epsfig,amsbsy,amsmath,amscd}
\numberwithin{equation}{section}
                        \theoremstyle{plain}
\usepackage{mathrsfs}

\newcommand\no[1]{}

\newtheorem{theorem}{Theorem}[section]

\newtheorem{lemma}[theorem]{Lemma}

\theoremstyle{definition}

\newtheorem{definition}[theorem]{Definition}

\newcommand{\BC}{\mathbb C}

\newcommand{\BZ}{\mathbb Z}
\newcommand{\BR}{\mathbb R}
\newcommand{\BT}{\mathbb T}

\newcommand{\BD}{\mathbb D}

\usepackage{tikz}
\usetikzlibrary{calc}
\usetikzlibrary{decorations.markings}
\usetikzlibrary{decorations.pathreplacing}
\usetikzlibrary{arrows,shapes,positioning}
\tikzstyle directed=[postaction={decorate,decoration={markings,
    mark=at position #1 with {\arrow{>}}}}]
\tikzstyle rdirected=[postaction={decorate,decoration={markings,
    mark=at position #1 with {\arrow{<}}}}]
\tikzset{anchorbase/.style={baseline={([yshift=-0.5ex]current bounding box.center)}}}

\tikzset{
    partial ellipse/.style args={#1:#2:#3}{
        insert path={+ (#1:#3) arc (#1:#2:#3)}
    }
}

\newcommand{\be} { \begin{equation} }
\newcommand{\ee} { \end{equation} }

\usepackage{hyperref}

\makeatletter 
\def\l@subsection{\@tocline{2}{0pt}{2pc}{6pc}{}} \makeatother

\begin{document}

\title[The $\mathfrak{gl}_2$-Skein Module of Lens Spaces via the Torus and Solid Torus]
{The $\mathfrak{gl}_2$-Skein Module of Lens Spaces via the Torus and Solid Torus}

\author{Hoang-An Nguyen}  
\address{Department of Mathematics, The University of Iowa, 
Iowa City, Iowa 52242, USA}
\email{hnguyen19@uiowa.edu}


\begin{abstract}
We compute the action of the $\mathfrak{gl}_2$-skein algebra of the torus on the $\mathfrak{gl}_2$-skein module of the solid torus. As a result, we show that the $\mathfrak{gl}_2$-skein modules of lens spaces is spanned by $\left(\left\lfloor{\frac{p}{2}}\right \rfloor+1\right)\left(2\left\lfloor{\frac{p}{2}}\right \rfloor+1\right)$ elements.
\end{abstract}


\thanks{{\it Key words and phrases.\/} .}

\maketitle



\tableofcontents

    
    
    
    
    
    

\section{Introduction}
\subsection{Motivation}

\subsubsection{Kauffman Bracket Skein Module, Torus, Solid Torus, Lens Spaces}

A fundamental invariant of oriented $3$-manifolds $M$ emerging from quantum topology is the "Kauffman bracket skein module" $S_A(M)$ introduced by \cite{Prz91, Tur88}. This is defined as the $\BC[A,A^{-1}]$-module formally spanned by all framed links in $M$, modulo isotopy equivalence and the linear relations
\begin{center}
\begin{tikzpicture}[anchorbase, scale=.5]
\draw [very thick] (2,1) to [out=180,in=0] (0,0);
 \draw [white,line width=.15cm] (2,0) to [out=180,in=0] (0,1) ;
\draw [very thick] (2,0) to [out=180,in=0] (0,1);
\end{tikzpicture} 
\;
$-$
\; A \;
\begin{tikzpicture}[anchorbase, scale=.5]
\draw [very thick] (2,1) to (0,1);
\draw [very thick] (2,0) to (0,0);
\end{tikzpicture}
\;$-$\;
$A^{-1}$ \;
\begin{tikzpicture}[anchorbase, scale=.5]
\draw [very thick] (2,0) to[out=180,in=315] (1.3,.5);
\draw [very thick] (2,1) to[out=180,in=45] (1.3,.5);
\draw [very thick] (.7,.5) to[out=135,in=0]  (0,1);
\draw [very thick] (.7,.5) to[out=225,in=0] (0,0);
\end{tikzpicture}
\end{center}
\begin{center}
\begin{tikzpicture}[fill opacity=.2,anchorbase,scale=.3]
\draw[very thick] (1,0) to [out=0,in=270] (2,1) to [out=90,in=0] (1,2)to [out=180,in=90] (0,1)to [out=270,in=180] (1,0);
\end{tikzpicture}    
\; $+$\; $(A^2+A^{-2})\emptyset$,
\end{center}
where the links in each expression are identical except in a small ball in which they look like the picture above.

In the case that $M=F\times [0,1]$, a surface cross an interval, the Kauffman bracket skein module $S_A(F\times I)$ (often denoted as $S_A(F)$) has an algebra structure via stacking of skeins on top of one another using the interval part of the manifold, i.e. $s*s'$ is the result of placing $s$ atop of $s'$.

In \cite{FG00} they study the Kauffman bracket skein algebra of the torus and are able to give an explicit presentation of the algebra. Using this presentation they study the skein module of the solid torus as a module over the skein algebra of the torus and are able to describe its action. Such an action is possible because the boundary of the solid torus is the torus, this action is defined in Section \ref{secAct}.

After this they are able to find a finite spanning set for the Kauffman bracket skein modules of lens spaces. This is done by expressing the Kauffman bracket skein module of a lens space as
$$S_A(L(p,q))= S_A(S^1\times \BD)\bigotimes_{S_A(\BT)}S_A(S^1\times 
\BD).$$
Here $\BD$ denotes the closed unit disk, $S^1$ denotes the circle, and $\BT:=S^1\times S^1$ denotes the torus.

\subsubsection{Extending results to the case of $\mathfrak{gl}_2$}
In this paper, we follow \cite{FG00} but in the case of the $\mathfrak{gl}_2$-skein algebra and arrive at analogous results. Namely, we compute the action of the $\mathfrak{gl}_2$-skein algebra of the torus $G_t(\BT)$ on the $\mathfrak{gl}_2$-skein module of the solid torus $G_t(S^1\times \BD)$. And as a result of this, we are also able to find a finite spanning set for $\mathfrak{gl}_2$-skein modules of lens spaces anaologous to the one found in \cite{FG00}.

The $\mathfrak{gl}_2$-skein algebra was introduced in \cite{QW18} and can be described in analogy to the Kauffman bracket skein module. Except it now consists of webs with labelled edges and crossing relations such as\\
\begin{gather}
\nonumber
\begin{tikzpicture}[anchorbase, scale=.5]
\draw [very thick, ->] (2,1) to [out=180,in=0] (0,0);
 \draw [white,line width=.15cm] (2,0) to [out=180,in=0] (0,1) ;
\draw [very thick, ->] (2,0) to [out=180,in=0] (0,1);
\end{tikzpicture}
 \;\;:= \;\;
\begin{tikzpicture}[anchorbase, scale=.5]
\draw [very thick, ->] (2,1) to (0,1);
\draw [very thick, ->] (2,0) to (0,0);
\end{tikzpicture}
\;\;-\;\;
 t\;
\begin{tikzpicture}[anchorbase, scale=.5]
\draw [very thick] (2,0) to[out=180,in=315] (1.3,.5);
\draw [very thick] (2,1) to[out=180,in=45] (1.3,.5);
\draw [double] (1.3,.5) -- (.7,.5);
\draw [very thick, ->] (.7,.5) to[out=135,in=0]  (0,1);
\draw [very thick, ->] (.7,.5) to[out=225,in=0] (0,0);
\end{tikzpicture}
\\ \nonumber
\begin{tikzpicture}[anchorbase, scale=.5]
\draw [very thick, ->] (2,0) to [out=180,in=0] (0,1);
 \draw [white,line width=.15cm] (2,1) to [out=180,in=0] (0,0) ;
\draw [double, ->] (2,1) to [out=180,in=0] (0,0);
\end{tikzpicture}
\;\;:= \;\;
-t^{-1}   \;
\begin{tikzpicture}[anchorbase, scale=.5]
\draw [double] (2,1) -- (1.4,1);
\draw [very thick, ->] (1.4,1) -- (0,1);
\draw [very thick] (2,0) -- (.6,0);
\draw [double, ->] (0.6,0) -- (0,0);
\draw [very thick] (.6,0) -- (1.4,1);
\end{tikzpicture} 
\quad,\quad
\begin{tikzpicture}[anchorbase, scale=.5]
\draw [double, ->] (2,0) to [out=180,in=0] (0,1);
 \draw [white,line width=.15cm] (2,1) to [out=180,in=0] (0,0) ;
\draw [double, ->] (2,1) to [out=180,in=0] (0,0);
\end{tikzpicture}
\;\;:= \;\; t^{-2} \;
\begin{tikzpicture}[anchorbase, scale=.5]
\draw [double, ->] (2,0) -- (0,0);
\draw [double, ->] (2,1) -- (0,1);
\end{tikzpicture}
\quad .
\end{gather}\\
\noindent
The full definition of $\mathfrak{gl_2}$-skein modules can be found in Section \ref{secgl2defns}. It should be noted that in \cite{QW18} they were able to find an explicit presentation for $G_t(\BT)$ in analogy to the presentation for the Kauffman bracket skein algebra of the torus found in \cite{FG00}. Namely, they find a basis consisting of elements $(m,n)_T$ and $\wedge^{(r,s)}$ which can be visualised by labelled multicurves on the torus. And most importantly, they find the relation 
$$(m,n)_T*(r,s)_T=(m+r,n+s)_T+(m-r,n-s)_T*\wedge^{(r,s)}$$
which allows us to do computations. This formula is called the Frohman-Gelca formula in \cite{QW18}.
See Section \ref{secgl2defns} for full details.

It may also be possible to view the $\mathfrak{gl_2}$-skein module in the framework of factorization homology following \cite{BBJ18}. This was explicitly done for the Kauffman bracket skein module in \cite{Coo19}.

\subsection{Main Results}
We now state the main results of the paper. Whose proofs are shown in Section \ref{secProofs}.
In the theorems below, $(m,n)_T$ and $\wedge^{(r,s)}$ are basis elements of $G_t(\BT)$ and can be visualised as labelled multicurves on the torus. Similarly, $(k)_T$ and $\wedge^{(l)}$ are basis elements of $G_t(S^1\times \BD)$ and can be visualised as multicurves inside the solid torus. For full details see Section \ref{secgl2defns}.

The first theorem gives an explicit formula describing the action of $\mathfrak{gl_2}$-skein algebra of the torus $G_t(\BT)$ on the $\mathfrak{gl_2}$-skein module of the solid torus $G_t(S^1\times \BD)$. 
\begin{theorem}\label{actthm}
(Action on Solid Torus)

\noindent
Let $x_{m,n}$ denote the image of $(m,n)_T\in G_t(\BT)$ under the projection map $\pi:G_t(\BT)\rightarrow G_t(S^1\times \BD)$, and $"\cdot"$ denote the canonical left action of $G_t(\BT)$ on $G_t(S^1\times \BD)$. We then have that  
\begin{align}\label{actform}
\begin{split}
(m,n)_T\cdot[(k)_T*\wedge^{(l)}]&=t^{-2nl} \wedge^{(l)}* x_{m+k,n}+t^{-2n(k+l)}\wedge^{(k+l)}* x_{m-k,n}\\
\wedge^{(m,n)}\cdot[(k)_T*\wedge^{(l)}]&=t^{-2n(m+k+2l)}(k)_T*\wedge^{(l+m)}.
\end{split}
\end{align}\vspace{1mm}
\end{theorem}

The following theorem gives a spanning set for $\mathfrak{gl_2}$-skein modules of lens spaces when viewed as the tensor product $G_T(S^1\times \BD)\otimes_{G_t(\BT)}G_t(S^1\times \BD)$

\begin{theorem}\label{lenthm}
(Spanning set for $\mathfrak{gl_2}$-skein Module of Lens Spaces)

\noindent
For the lens space $L(p,q)$, we have that $G_t(L(p,q))$ as a $\BC(t)$-vector space is spanned by the elements
$$\{(n)*\wedge^{(m)}\otimes 1\},$$
where $n\in\{0,\dots,\left\lfloor{\frac{p}{2}}\right \rfloor\}$ and $m\in \{-\left\lfloor{\frac{p}{2}}\right \rfloor,\dots ,\left\lfloor{\frac{p}{2}}\right \rfloor\}$ with $\left\lfloor{\frac{p}{2}}\right \rfloor$ meaning the greatest integer less than or equal to $\frac{p}{2}$.
\end{theorem}

\subsection{Outline} The paper is outlined as follows.\\

In Section \ref{secgl2defns} we go over prerequisite background. More specifically, we recall definitions of $G_t(M)$. After this we then look at the case of a 3-manifold $M$ with boundary $\partial M$, and look at the corresponding action of $G_t(\partial M)$ on $G_t(M)$. We then focus to the case of the solid torus and the torus.

In Section \ref{secAct} we find formulas for the projection of elements in $G_t(\BT)$ to $G_t(S^1\times \BD)$, which are used in the proof of Theorem \ref{actthm}.

In Section \ref{secLens} we recall Heegaard splittings of 3-manifolds and how this interacts with skein theory. We then look specifically at Heegaard splittings of lens spaces and its corresponding $\mathfrak{gl_2}$-skein module.

In Section \ref{secProofs} we establish the main lemmas \ref{mainlemma} and \ref{lexpression}. And then use these lemmas to finally prove Theorem \ref{lenthm}.

\section*{Acknowledgements} 
The author would like to thank Charles Frohman and Ben Cooper for their patient support and guidance while working on this problem. And especially so, to Ben Cooper for his time and advice while the author was making revisions to drafts of this paper.

\section{Preliminary Material for $\mathfrak{gl_2}$-Skein Theory}
\label{secgl2defns}
\subsection{$\mathfrak{gl_2}$-Skein Modules and Algebras} We introduce the definitions regarding the $\mathfrak{gl_2}$-skein module. Here we will be following \cite{QW18, FS20}.

\begin{definition} 
A \textbf{marked surface} is an oriented surface $F$ together with a finite subset $P$ of $\partial F$ ($\partial F$ may be empty) of \textbf{marked points}.

Let $I=[0,1]$. A $\mathfrak{gl}_2$-\textbf{web} in $(F,P)\times I$ consists of
\begin{itemize}
\item 1-valent external vertices in fibers $\{p\}\times I$ over (some of) the marked points $p\in P$
    
\item(Points in the middle) 3-valent internal vertices of the form
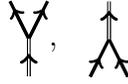
\begin{center}
\begin{tikzpicture}[anchorbase, scale=.5]
\draw [very thick] (1.5,1) -- (1,2);
\draw [very thick] (1.5,1) -- (2,2);
\draw [very thick, ->] (1.5,1) -- (1.1,1.8);
\draw [very thick, ->] (1.5,1) -- (1.9,1.8);
\draw [double] (1.5,0) -- (1.5,1);
\draw [double, ->] (1.5,0) -- (1.5,.6);
\end{tikzpicture},\quad
\begin{tikzpicture}[anchorbase, scale=.5]
\draw [very thick] (1,0) -- (1.5,1);
\draw [very thick] (2,0) -- (1.5,1);
\draw [very thick, ->] (1,0) -- (1.25,.5);
\draw [very thick, ->] (2,0) -- (1.75,.5);
\draw [double] (1.5,1) -- (1.5,2);
\draw [double, ->] (1.5,1) -- (1.5,1.75);
\end{tikzpicture}
\end{center}
    
\item 1-labelled oriented edges
\begin{center}
\begin{tikzpicture}[anchorbase, scale=.5]
\draw [very thick] (1.5,1) -- (1.5,2);
\draw [very thick, ->] (1.5,1) -- (1.5,1.75);
\end{tikzpicture}    
\end{center}

\item 2-labelled oriented edges
\begin{center}
\begin{tikzpicture}[anchorbase, scale=.5]
\draw [double] (1.5,1) -- (1.5,2);
\draw [double, ->] (1.5,1) -- (1.5,1.75);
\end{tikzpicture}    
\end{center}

\item 1-labelled oriented loops
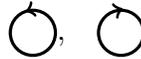
\begin{center}
\begin{tikzpicture}[fill opacity=.2,anchorbase,scale=.3]
\draw[very thick, directed=.55] (1,0) to [out=0,in=270] (2,1) to [out=90,in=0] (1,2)to [out=180,in=90] (0,1)to [out=270,in=180] (1,0);
\end{tikzpicture},\quad
\begin{tikzpicture}[fill opacity=.2,anchorbase,scale=.3]
\draw[very thick, rdirected=.55] (1,0) to [out=0,in=270] (2,1) to [out=90,in=0] (1,2)to [out=180,in=90] (0,1)to [out=270,in=180] (1,0);
\end{tikzpicture}
\end{center}

\item 2-labelled oriented loops
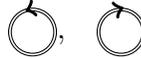
\begin{center}
\begin{tikzpicture}[fill opacity=.2,anchorbase,scale=.3]
\draw[double, directed=.55] (1,0) to [out=0,in=270] (2,1) to [out=90,in=0] (1,2)to [out=180,in=90] (0,1)to [out=270,in=180] (1,0);
\end{tikzpicture},\quad
\begin{tikzpicture}[fill opacity=.2,anchorbase,scale=.3]
\draw[double, rdirected=.55] (1,0) to [out=0,in=270] (2,1) to [out=90,in=0] (1,2)to [out=180,in=90] (0,1)to [out=270,in=180] (1,0);
\end{tikzpicture}
\end{center}

\item 
A framing, which can be formally defined by
    \begin{itemize}
    \item a choice of a properly embedded oriented compact surface $S$ having the web as a spine, called framing, such that the projection on the first factor $p:F\times I\rightarrow F$ restricts to a local homeomorphism $p|_S:S\rightarrow F$.
    \item 
    $S\cap \partial F\times I$ is a collection of horizontal arcs, each of which containing a unique point of $M$.
    \end{itemize}
\end{itemize}
A web may be empty, $\emptyset$, or disconnected. Webs are considered up to isotopy.

A $\mathfrak{gl}_2$-\textbf{tangled web} is a $\mathfrak{gl}_2$-web with the additional information of 4-valent vertices of the form
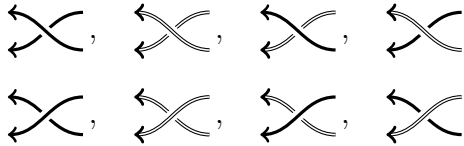
\begin{center}
\begin{tikzpicture}[anchorbase, scale=.5]
\draw [very thick, ->] (2,1) to [out=180,in=0] (0,0);
 \draw [white,line width=.15cm] (2,0) to [out=180,in=0] (0,1) ;
\draw [very thick, ->] (2,0) to [out=180,in=0] (0,1);
\end{tikzpicture},\quad
\begin{tikzpicture}[anchorbase, scale=.5]
\draw [double, ->] (2,1) to [out=180,in=0] (0,0);
 \draw [white,line width=.15cm] (2,0) to [out=180,in=0] (0,1) ;
\draw [double, ->] (2,0) to [out=180,in=0] (0,1);
\end{tikzpicture},\quad
\begin{tikzpicture}[anchorbase, scale=.5]
\draw [double, ->] (2,1) to [out=180,in=0] (0,0);
 \draw [white,line width=.15cm] (2,0) to [out=180,in=0] (0,1) ;
\draw [very thick, ->] (2,0) to [out=180,in=0] (0,1);
\end{tikzpicture},\quad
\begin{tikzpicture}[anchorbase, scale=.5]
\draw [very thick, ->] (2,1) to [out=180,in=0] (0,0);
 \draw [white,line width=.15cm] (2,0) to [out=180,in=0] (0,1) ;
\draw [double, ->] (2,0) to [out=180,in=0] (0,1);
\end{tikzpicture}
\end{center}
\vspace{2mm}
\begin{center}
\begin{tikzpicture}[anchorbase, scale=.5]
\draw [very thick, ->] (2,0) to [out=180,in=0] (0,1);
 \draw [white,line width=.15cm] (2,1) to [out=180,in=0] (0,0) ;
\draw [very thick, ->] (2,1) to [out=180,in=0] (0,0);
\end{tikzpicture},\quad
\begin{tikzpicture}[anchorbase, scale=.5]
\draw [double, ->] (2,0) to [out=180,in=0] (0,1);
 \draw [white,line width=.15cm] (2,1) to [out=180,in=0] (0,0) ;
\draw [double, ->] (2,1) to [out=180,in=0] (0,0);
\end{tikzpicture},\quad
\begin{tikzpicture}[anchorbase, scale=.5]
\draw [double, ->] (2,0) to [out=180,in=0] (0,1);
 \draw [white,line width=.15cm] (2,1) to [out=180,in=0] (0,0) ;
\draw [very thick, ->] (2,1) to [out=180,in=0] (0,0);
\end{tikzpicture},\quad
\begin{tikzpicture}[anchorbase, scale=.5]
\draw [very thick, ->] (2,0) to [out=180,in=0] (0,1);
 \draw [white,line width=.15cm] (2,1) to [out=180,in=0] (0,0) ;
\draw [double, ->] (2,1) to [out=180,in=0] (0,0);
\end{tikzpicture}
\end{center}
\end{definition}

In this paper we consider consider oriented surfaces $\Sigma$ with no marked points. Namely, the webs have no 1-valent vertices, and hence do not start and end anywhere, as in the case of the original Kauffman bracket skein module.
\begin{definition}
We let $G_t(\Sigma)$, the $\mathfrak{gl_2}$-skein algebra of $\Sigma\times I$, denote the quotient of the free $\BZ[t^{\pm 1}]$-module spanned by isotopy classes of tangled $\mathfrak{gl}_2$-webs embedded in $F\times I$, by the ideal generated by local relations
\begin{gather}
\label{eqn:circles}
\begin{tikzpicture}[fill opacity=.2,anchorbase,scale=.3]
\draw[very thick, directed=.55] (1,0) to [out=0,in=270] (2,1) to [out=90,in=0] (1,2)to [out=180,in=90] (0,1)to [out=270,in=180] (1,0);
\end{tikzpicture} 
\quad=\quad
(t+ t^{-1}) \emptyset
\quad=\quad 
\begin{tikzpicture}[fill opacity=.2,anchorbase,scale=.3]
\draw[very thick, rdirected=.55] (1,0) to [out=0,in=270] (2,1) to [out=90,in=0] (1,2)to [out=180,in=90] (0,1)to [out=270,in=180] (1,0);
\end{tikzpicture}
\quad,\quad
\begin{tikzpicture}[fill opacity=.2,anchorbase,scale=.3]
\draw[double, directed=.55] (1,0) to [out=0,in=270] (2,1) to [out=90,in=0] (1,2)to [out=180,in=90] (0,1)to [out=270,in=180] (1,0);
\end{tikzpicture} 
\quad=\quad
\emptyset
\quad=\quad \begin{tikzpicture}[fill opacity=.2,anchorbase,scale=.3]
\draw[double, rdirected=.55] (1,0) to [out=0,in=270] (2,1) to [out=90,in=0] (1,2)to [out=180,in=90] (0,1)to [out=270,in=180] (1,0);
\end{tikzpicture}
\\
\label{eqn:digons}
\begin{tikzpicture}[anchorbase, scale=.5]
\draw [double] (.5,0) -- (.5,.3);
\draw [very thick] (.5,.3) .. controls (.4,.35) and (0,.6) .. (0,1) .. controls (0,1.4) and (.4,1.65) .. (.5,1.7);
\draw [very thick] (.5,.3) .. controls (.6,.35) and (1,.6) .. (1,1) .. controls (1,1.4) and (.6,1.65) .. (.5,1.7);
\draw [double, ->] (.5,1.7) -- (.5,2);
\end{tikzpicture}
\quad= \quad
(t+t^{-1})\;
\begin{tikzpicture}[anchorbase, scale=.5]
\draw [double,->] (.5,0) -- (.5,2);
\end{tikzpicture}
\quad,\quad
\begin{tikzpicture}[anchorbase, scale=.5]
\draw [very thick] (.5,0) -- (.5,.3);
\draw [very thick] (.5,.3) .. controls (.4,.35) and (0,.6) .. (0,1) .. controls (0,1.4) and (.4,1.65) .. (.5,1.7);
\draw [double, directed=0.55] (.5,.3) .. controls (.6,.35) and (1,.6) .. (1,1) .. controls (1,1.4) and (.6,1.65) .. (.5,1.7);
\draw [very thick, ->] (.5,1.7) -- (.5,2);
\end{tikzpicture}
\quad= \quad
\begin{tikzpicture}[anchorbase, scale=.5]
\draw [very thick,->] (.5,0) -- (.5,2);
\end{tikzpicture}
\quad= \quad
\begin{tikzpicture}[anchorbase, scale=.5]
\draw [very thick] (.5,0) -- (.5,.3);
\draw [double, directed=0.55] (.5,.3) .. controls (.4,.35) and (0,.6) .. (0,1) .. controls (0,1.4) and (.4,1.65) .. (.5,1.7);
\draw [very thick] (.5,.3) .. controls (.6,.35) and (1,.6) .. (1,1) .. controls (1,1.4) and (.6,1.65) .. (.5,1.7);
\draw [very thick, ->] (.5,1.7) -- (.5,2);
\end{tikzpicture}
\\
\label{eqn:squares}
\begin{tikzpicture}[anchorbase,scale=.5]
\draw [double] (0,0) -- (0,0.5);
\draw [very thick] (1,0) -- (1,.7);
\draw [very thick] (0,0.5) -- (1,.7);
\draw [double] (1,.7) -- (1,1.3);
\draw [very thick] (0,.5) -- (0,1.5);
\draw [very thick] (1,1.3) -- (0,1.5);
\draw [double,->] (0,1.5) -- (0,2);
\draw [very thick, ->] (1,1.3) -- (1,2);
\end{tikzpicture}
\quad = \quad
\begin{tikzpicture}[anchorbase,scale=.5]
\draw [double,->] (0,0) -- (0,2);
\draw [very thick,->] (1,0) -- (1,2);
\end{tikzpicture}
\quad,\quad
\begin{tikzpicture}[anchorbase,scale=.5]
\draw [double] (1,0) -- (1,0.5);
\draw [very thick] (0,0) -- (0,.7);
\draw [very thick] (1,0.5) -- (0,.7);
\draw [double] (0,.7) -- (0,1.3);
\draw [very thick] (1,.5) -- (1,1.5);
\draw [very thick] (0,1.3) -- (1,1.5);
\draw [double,->] (1,1.5) -- (1,2);
\draw [very thick, ->] (0,1.3) -- (0,2);
\end{tikzpicture}
\quad = \quad
\begin{tikzpicture}[anchorbase,scale=.5]
\draw [double,->] (1,0) -- (1,2);
\draw [very thick,->] (0,0) -- (0,2);
\end{tikzpicture}
\quad , \quad
\begin{tikzpicture}[anchorbase,scale=.5]
\draw [double,->] (0,0) to  (0,2);
\draw [double,->] (1,2) to (1,0);
\end{tikzpicture}
\quad =\quad
\begin{tikzpicture}[anchorbase,scale=.5]
\draw [double,->] (0,0) to (0,.5) to [out=90,in=90] (1,.5) to (1,0);
\draw [double,->] (1,2) to (1,1.5) to [out=270,in=270] (0,1.5) to (0,2);
\end{tikzpicture}
\quad , \quad
\begin{tikzpicture}[anchorbase,scale=.5]
\draw [double,<-] (0,0) to  (0,2);
\draw [double,<-] (1,2) to (1,0);
\end{tikzpicture}
\quad =\quad
\begin{tikzpicture}[anchorbase,scale=.5]
\draw [double,<-] (0,0) to (0,.5) to [out=90,in=90] (1,.5) to (1,0);
\draw [double,<-] (1,2) to (1,1.5) to [out=270,in=270] (0,1.5) to (0,2);
\end{tikzpicture}
\quad .
\end{gather}
and local crossing relations
\begin{gather}
\begin{tikzpicture}[anchorbase, scale=.5]
\draw [very thick, ->] (2,1) to [out=180,in=0] (0,0);
 \draw [white,line width=.15cm] (2,0) to [out=180,in=0] (0,1) ;
\draw [very thick, ->] (2,0) to [out=180,in=0] (0,1);
\end{tikzpicture}
 \;\;:= \;\;
\begin{tikzpicture}[anchorbase, scale=.5]
\draw [very thick, ->] (2,1) to (0,1);
\draw [very thick, ->] (2,0) to (0,0);
\end{tikzpicture}
\;\;-\;\;
 t\;
\begin{tikzpicture}[anchorbase, scale=.5]
\draw [very thick] (2,0) to[out=180,in=315] (1.3,.5);
\draw [very thick] (2,1) to[out=180,in=45] (1.3,.5);
\draw [double] (1.3,.5) -- (.7,.5);
\draw [very thick, ->] (.7,.5) to[out=135,in=0]  (0,1);
\draw [very thick, ->] (.7,.5) to[out=225,in=0] (0,0);
\end{tikzpicture}
\quad,\quad
\begin{tikzpicture}[anchorbase, scale=.5]
\draw [very thick, ->] (2,0) to [out=180,in=0] (0,1);
 \draw [white,line width=.15cm] (2,1) to [out=180,in=0] (0,0) ;
\draw [very thick, ->] (2,1) to [out=180,in=0] (0,0);
\end{tikzpicture}
\;\;:=\;\;
\begin{tikzpicture}[anchorbase, scale=.5]
\draw [very thick, ->] (2,1) to (0,1);
\draw [very thick, ->] (2,0) to (0,0);
\end{tikzpicture}
\;\;-\;\;
t^{-1}   \;
\begin{tikzpicture}[anchorbase, scale=.5]
\draw [very thick] (2,0) to[out=180,in=315] (1.3,.5);
\draw [very thick] (2,1) to[out=180,in=45] (1.3,.5);
\draw [double] (1.3,.5) -- (.7,.5);
\draw [very thick, ->] (.7,.5) to[out=135,in=0]  (0,1);
\draw [very thick, ->] (.7,.5) to[out=225,in=0] (0,0);
\end{tikzpicture}
\nonumber
\\
\label{eq:crossing}
\begin{tikzpicture}[anchorbase, scale=.5]
\draw [very thick, ->] (2,1) to [out=180,in=0] (0,0);
 \draw [white,line width=.15cm] (2,0) to [out=180,in=0] (0,1) ;
\draw [double, ->] (2,0) to [out=180,in=0] (0,1);
\end{tikzpicture}
\;\;:= \;\;-t \;\;\,
\begin{tikzpicture}[anchorbase, scale=.5]
\draw [double] (2,0) -- (1.4,0);
\draw [very thick, ->] (1.4,0) -- (0,0);
\draw [very thick] (2,1) -- (.6,1);
\draw [double, ->] (0.6,1) -- (0,1);
\draw [very thick] (.6,1) -- (1.4,0);
\end{tikzpicture} 
\quad,\quad
\begin{tikzpicture}[anchorbase, scale=.5]
\draw [double, ->] (2,1) to [out=180,in=0] (0,0);
 \draw [white,line width=.15cm] (2,0) to [out=180,in=0] (0,1) ;
\draw [very thick, ->] (2,0) to [out=180,in=0] (0,1);
\end{tikzpicture}
\;\;:= \;\;-t  \;\;\,
\begin{tikzpicture}[anchorbase, scale=.5]
\draw [double] (2,1) -- (1.4,1);
\draw [very thick, ->] (1.4,1) -- (0,1);
\draw [very thick] (2,0) -- (.6,0);
\draw [double, ->] (0.6,0) -- (0,0);
\draw [very thick] (.6,0) -- (1.4,1);
\end{tikzpicture} 
\quad,\quad
\begin{tikzpicture}[anchorbase, scale=.5]
\draw [double, ->] (2,1) to [out=180,in=0] (0,0);
 \draw [white,line width=.15cm] (2,0) to [out=180,in=0] (0,1) ;
\draw [double, ->] (2,0) to [out=180,in=0] (0,1);
\end{tikzpicture}
\;\;:= \;\;t^{2}  \;\;\,
\begin{tikzpicture}[anchorbase, scale=.5]
\draw [double, ->] (2,0) -- (0,0);
\draw [double, ->] (2,1) -- (0,1);
\end{tikzpicture}
\\ \nonumber
\begin{tikzpicture}[anchorbase, scale=.5]
\draw [very thick, ->] (2,0) to [out=180,in=0] (0,1);
 \draw [white,line width=.15cm] (2,1) to [out=180,in=0] (0,0) ;
\draw [double, ->] (2,1) to [out=180,in=0] (0,0);
\end{tikzpicture}
\;\;:= \;\;
-t^{-1}   \;
\begin{tikzpicture}[anchorbase, scale=.5]
\draw [double] (2,1) -- (1.4,1);
\draw [very thick, ->] (1.4,1) -- (0,1);
\draw [very thick] (2,0) -- (.6,0);
\draw [double, ->] (0.6,0) -- (0,0);
\draw [very thick] (.6,0) -- (1.4,1);
\end{tikzpicture} 
\quad,\quad
\begin{tikzpicture}[anchorbase, scale=.5]
\draw [double, ->] (2,0) to [out=180,in=0] (0,1);
 \draw [white,line width=.15cm] (2,1) to [out=180,in=0] (0,0) ;
\draw [very thick, ->] (2,1) to [out=180,in=0] (0,0);
\end{tikzpicture}
\;\;:= \;\;
- t^{-1} \;
\begin{tikzpicture}[anchorbase, scale=.5]
\draw [double] (2,0) -- (1.4,0);
\draw [very thick, ->] (1.4,0) -- (0,0);
\draw [very thick] (2,1) -- (.6,1);
\draw [double, ->] (0.6,1) -- (0,1);
\draw [very thick] (.6,1) -- (1.4,0);
\end{tikzpicture} 
\quad,\quad
\begin{tikzpicture}[anchorbase, scale=.5]
\draw [double, ->] (2,0) to [out=180,in=0] (0,1);
 \draw [white,line width=.15cm] (2,1) to [out=180,in=0] (0,0) ;
\draw [double, ->] (2,1) to [out=180,in=0] (0,0);
\end{tikzpicture}
\;\;:= \;\; t^{-2} \;
\begin{tikzpicture}[anchorbase, scale=.5]
\draw [double, ->] (2,0) -- (0,0);
\draw [double, ->] (2,1) -- (0,1);
\end{tikzpicture}
\quad .
\end{gather}
The diagrams above are from [QW 18]. For convenience, when orientation of an edge can be determined by the diagram we do not draw the arrows, as in the case of the diagrams above.

$G_t(\Sigma)$ has an algebra structure in the same way the Kauffman bracket skein algebra does, namely $w*w'$ is the result of stacking $w$ atop of $w'$ via the interval part of $\Sigma\times I$. In this paper, we use the term "skein" to mean an element of $G_t(\Sigma)$.
\end{definition}

The skein theory presented here encodes the pivotal tensor category of representations of $U_t(\mathfrak{gl}_2)$ generated by the vector representation $V$ and its exterior power $\wedge^{2}V$. In words, the 1-labeled edges represent the vector representation of $U_t(\mathfrak{gl}_2)$ and the 2-labeled edges represent the exterior square of the vector representation.\\

One defines $\mathfrak{gl_2}$-skein modules for 3-manifolds in a similar fashion, where the relations take place in a local 3-ball where one can choose a projection in the 3-ball. In this paper we only consider $\mathfrak{gl}_2$-webs in our manifold that do not have 1-valent vertices. This is in analogy of the Kauffman bracket skein module of 3-manifolds.

It should be noted that there is a module map between $G_t(\Sigma)$ and the Kauffman bracket skein algebra $S_A(\Sigma)$ that is described in \cite{QW18}. We refer to \cite{QW18} for more details.

\subsection{$\mathfrak{gl_2}$-Skeins and 3-Manifolds with Boundary} We now look at $G_t(M)$ for a 3-manifold $M$ with boundary $\partial M$, and a canonical action of $G_t(\partial M)$ on $G_t(M)$.

Note that there is an embedding map $i_M:\partial M\rightarrow 
 M$ which induces the map
\begin{align*}
\pi_M:G_t(\partial M)&\rightarrow G_t(M)\\
A &\mapsto \Tilde{i}_M(A),
\end{align*}
where $\Tilde{i}_M$ is $\BC(t)$-linearization of $i_M$. From $\pi_M$ we can define the following left action of $G_t(\partial M)$ on $G_t(M)$ by
$$A\cdot \alpha = \pi_M(A)\sqcup \alpha,$$
where $A\in G_t(\partial M)$ and $\alpha\in G_t(M)$.
In this paper we will call this the canonical left action of $G_t(\partial M)$ on $G_t(M)$.
It should also be noted that the projection map $\pi_M$ is not an algebra homomorphism, but we do have the following lemma.
\begin{lemma}
Let $A,B\in G_t(\partial M)$. Then we have
\begin{equation}
\label{projectionaction}
\pi_M(A* B)= A\cdot\pi_M(B).    
\end{equation}
Here $*$ on the left hand side denotes the multiplication in $G_t(\partial M)$, and $\cdot$ on the right hand side denotes the left action of $G_t(\partial M)$ on $G_t(M)$.
\end{lemma}
\begin{proof}
Using the definitions of the action and the fact the product in $G_t(\partial M)$ is defined via stacking of skeins, we see that
\begin{align*}
A\cdot \pi_M(B)&= \pi_M(A)\sqcup \pi_{M}(B)\\
&=\pi_M(A\sqcup B)\\
&=\pi_M(A*B)
\end{align*}
\end{proof}
\subsection{$\mathfrak{gl_2}$-Skein Algebra of the Torus and Solid Torus} We now look at the specific case when $M= S^1\times \BD$ and $\partial M=\BT$. In this section we will be following \cite{QW18}.\\

\noindent
\textbf{The Torus:}

Let $(m,n)$ denote the oriented framed 1-labeled curve in $\BT$ with homology class $(m,n)\in H_1(\BT)\cong \BZ\oplus \BZ$, and $\wedge^{(r,s)}$ denote the oriented framed 2-labeled curve with homology class $2(r,s)\in H_1(\BT)$ when viewing the $2$-labeled curve as two parallel curves.\\

Following \cite{QW18}, one basis for $G_t(\BT)$, called the standard basis, is
$$\{(m,n)*\wedge^{(r,s)}| m,n,r,s\in\BZ \}.$$
A more useful basis for $G_t(\BT)$ that we will use, also introduced in \cite{QW18}, is the standard $T$-basis, given by
$$\{(m,n)_T*\wedge^{(r,s)} | m,n,r,s\in\BZ, \hspace{2mm} m>0 \text{ or } n>m=0\}\cup \{\wedge^{(r,s)}|r,s\in\BZ\}$$
where $(m,n)_T$ is defined by induction on $gcd(m,n)$ as follows:\\
\begin{equation}\label{tbasis}
(m,n)_T=
\begin{cases}
(m,n), & \gcd(m,n)=1\\
(m,n)-2\wedge^{(m/2,n/2)}, & \gcd(m,n)=2\\
(m-a,n-b)_T*(a,b)-(m-2a,n-2b)_T*\wedge^{(a,b)}, & \gcd(m,n)=d\geq 3,\\ &\text{ and } (m,n)=(da,db)
\end{cases}.
\end{equation}

This basis is analogous to the one used in \cite{FG00}. The T-basis can be defined more generally for $G_t(\Sigma)$, but involves introducing some technical definitions such as \textit{laminations} that are not needed in this paper. We refer the reader to \cite{QW18} if they are interested.
\begin{theorem}\cite{QW18}
The $\BZ[t^{\pm 1}]$-algebra $G_t(\BT)$ is isomorphic to the abstract $\BZ[t^{\pm 1}]$-algebra with generators $(m,n)_T$ and $\wedge^{(m,n)}$ with $(m,n)\in \BZ^2$, subject to the following relations:
\begin{align}
\label{FGform}
(m,n)_T*(r,s)_T&=(m+r,n+s)_T+(m-r,n-s)_T*\wedge^{(r,s)}\\
\label{switch}
(m,n)_T*\wedge^{(r,s)}&=t^{2(ms-nr)}\wedge^{(r,s)}*(m,n)_T\\
\label{revorien}
(m,n)_T*\wedge^{(-m,-n)}&=(-m,-n)_T\\
\label{wedgecombinesplit}
\wedge^{(m,n)}*\wedge^{(r,s)}&=t^{2(ms-nr)}\wedge^{(m+r,n+s)}\\
\label{0guys}
(0,0)_T=2,&\hspace{2mm} \wedge^{(0,0)}=1
\end{align}
\end{theorem}
\noindent
The first formula is in analogy of the product-to-sum formula found in \cite{FG00}, and is called the Frohman-Gelca formula in \cite{QW18}. The theorem as a whole gives a presentation of the algebra $G_t(\BT)$ which is analogous to the way a presentation was given to the Kauffman bracket skein algebra of the torus in \cite{FG00}. With this presentation of $G_t(\BT)$, we are able to do many concrete computations regarding $G_t(\BT)$.\\

\noindent
\textbf{Solid Torus:}

Let $(m)$ denote the oriented framed 1-labeled curve in $\BT$ with homology class $m\in H_1(S^1\times \BD)\cong \BZ$, and $\wedge^{(r)}$ denote the oriented framed 2-labeled curve with homology class $2r\in H_1(S^1\times \BD)$ when viewing the $2$-labeled curve as two parallel curves. Note that the solid torus can be viewed as an annulus cross an interval, and thus we have an algebra structure.\\

From \cite{QW18}, a basis for $G_t(S^1\times \BD)$ is
$$\{(n)*\wedge^{(r)}| n,r\in \BZ \}.$$
Since the basis elements are just cores of the solid torus or annulus, we see that the elements in $G_t(S^1\times \BD)$ can be isotoped to not intersect one another. This tells us that the the algebra is commutative and that we actually have the follwing simple relations,
\begin{align}
\label{solidezrelations}
\begin{split}
(n)*(m)&=(n+m)\\
\wedge^{(k)}*\wedge^{(l)}&=\wedge^{(k+l)}.
\end{split}
\end{align}

Another useful relation that will be used is the following.
\begin{lemma}\cite{QW18} In $G_t(S^1\times \BD)$, we have that
\begin{align}
\label{revorien2}
(n)*\wedge^{(-n)}=(-n)    
\end{align}
for all $n\in\BZ$.
\end{lemma}

At this point we have covered the definitions and the notation necessary for Theorem \ref{actthm} and Theorem \ref{lenthm}.

\section{Action of $G_t(\BT)$ on $G_t(S^1\times \BD)$}
\label{secAct}

\subsection{Finding the Projection of Skeins in the Torus onto the Solid Torus}
\label{ProjNotation}
We now explicitly compute the projection of elements of $G_t(\BT)$ onto $G_t(S^1\times \BD)$. This will be used to express the action of $G_t(\BT)$ on $G_t(S^1\times \BD)$.

For the following lemmas we let $\pi:G_t(\BT)\rightarrow G_t(S^1\times \BD)$ be the projection map induced from $i:\BT\rightarrow S^1\times \BD$. And for convenience we introduce the following notation:
\begin{align*}
y_{r,s}&:=\pi(\wedge^{(r,s)})\\
x_{m,n}&:=\pi((m,n)_T)\\
(m)_T&:=\pi((m,0)_T)).
\end{align*}
\begin{lemma}
For all $r,s\in\BZ$, we have
\label{projwedge}
\begin{align}
y_{r,s}=t^{-2rs}\wedge^{(r)}.    
\end{align}
\end{lemma}



\begin{proof}
Note from Equation \ref{wedgecombinesplit}, we have that
\begin{align*}
\wedge^{(r,s)}&=t^{-2rs}\wedge^{(r,0)}*\wedge^{(0,s)}.  
\end{align*}
And thus using the projection formula described by Equation \ref{projectionaction}, we see that
\begin{align*}
\pi(\wedge^{(r,s)}) &= t^{-2rs}\wedge^{(r,0)}\cdot \pi(\wedge^{(0,s)})\\
&= t^{-2rs}\wedge^{(r,0)}\cdot 1\\
&=t^{-2rs}\wedge^{(r)}.
\end{align*}
\end{proof}
We see that $y_{r,0}=\wedge^{(r)}$, and hence for convenience, instead of writing $y_{r,0}$ we will write $\wedge^{(r)}$ when applicable.

\begin{lemma}
\label{m0init}
For all $n\in \BZ$, we have
\begin{align}\label{m0initform}
x_{0,n}=(t)^n+(t^{-1})^n,    
\end{align}
\end{lemma}
\begin{proof}

We will use induction to show this.\\

(Base Case $n=1$):
One can see that 
$$x_{0,1}=t+t^{-1},$$
since the meridional loop now bounds a disk in the solid torus.\\

(Base Case $n=2$):
Note that
$$(0,2)_T=(0,2)-2\wedge^{(0,1)},$$
then applying the projection map we have
\begin{align*}
x_{0,2}&=(t+t^{-1})^2-2\pi(\wedge^{(0,1)})\\
&=t^2+t^{-2}
\end{align*}

(Induction hypothesis):
Suppose that we have,
\begin{align*}
x_{0,n}&=(t)^n+(t^{-1})^n\\    
x_{0,n-1}&=(t)^{n-1}+(t^{-1})^{n-1}.
\end{align*}
Note that 
$$(m,n)_T=(m-a,n-b)_T*(a,b)-(m-2a,n-2b)*\wedge^{(a,b)}$$ if $\gcd(m,n)=d\geq 3$ and $(m,n)=(da,db)$.\\

For us we have,
$\gcd(0,n+1)=n+1\geq 3$ and $(0,n+1)=((n+1)0,(n+1)(1))$. So we have $a=0,b=1,$ and $d=n+1$, which give us
\begin{align*}
(0,n+1)_T &= (0,(n+1)-1)_T*(0,1)-(0-2(0),(n+1)-2(1))_T*\wedge^{(0,1)}\\
&=(0,n)_T*(0,1)-(0,n-1)_T*\wedge^{(0,1)}.
\end{align*}
Now applying the projection map to this equation, we have
\begin{align*}
x_{0,n+1} &= [(t)^n+(t^{-1})^n](t+t^{-1})-[(t)^{n-1}+(t^{-1})^{n-1}]\\
&=(t)^{n+1}+(t^{-1})^{n-1} +t^{n-1}+(t^{-1})^{n+1}-[(t)^{n-1}+(t^{-1})^{n-1}]\\
&=(t)^{n+1}+(t^{-1})^{n+1}.
\end{align*}
One can analogously prove the result for negative $n$ using the same methods above.
\end{proof}

\begin{lemma}
\label{initialx}
For all $k\in \BZ$, we have
\begin{align}
\label{initialxform}
\begin{split}
x_{1,k}&=t^k(1)\\
x_{-1,k}&=t^{-k}(-1)
\end{split}
\end{align}
\end{lemma}
\begin{proof}
Note that $x_{1,k}$ is just the longitudinal element with $k$ twists. Locally each twist looks like the picture below for $k$ positive.
\begin{center}
\begin{tikzpicture}[anchorbase, scale=.4]
\draw [very thick, <-] (3,-1) to [out=180,in=-90] (0,1);
\draw [white,line width=.15cm] (2.05,0) to [out=180,in=0] (.05,-1) ;
\draw [very thick] (2,1) to [out=-90,in=0] (-1,-1);
\draw [very thick] (1,2) to [out=180,in=90] (0,1);
\draw [very thick] (1,2) to [out=0,in=90] (2,1);
\end{tikzpicture}    
\end{center}
Using the relations
\begin{center}
\begin{tikzpicture}[anchorbase, scale=.5]
\draw [very thick, ->] (2,0) to [out=180,in=0] (0,1);
 \draw [white,line width=.15cm] (2,1) to [out=180,in=0] (0,0) ;
\draw [very thick, ->] (2,1) to [out=180,in=0] (0,0);
\end{tikzpicture}
\;\;:=\;\;
\begin{tikzpicture}[anchorbase, scale=.5]
\draw [very thick, ->] (2,1) to (0,1);
\draw [very thick, ->] (2,0) to (0,0);
\end{tikzpicture}
\;\;-\;\;
$t^{-1}$   \;
\begin{tikzpicture}[anchorbase, scale=.5]
\draw [very thick] (2,0) to[out=180,in=315] (1.3,.5);
\draw [very thick] (2,1) to[out=180,in=45] (1.3,.5);
\draw [double] (1.3,.5) -- (.7,.5);
\draw [very thick, ->] (.7,.5) to[out=135,in=0]  (0,1);
\draw [very thick, ->] (.7,.5) to[out=225,in=0] (0,0);
\end{tikzpicture}
\quad,\quad
\begin{tikzpicture}[anchorbase, scale=.5]
\draw [very thick] (.5,0) -- (.5,.3);
\draw [very thick] (.5,.3) .. controls (.4,.35) and (0,.6) .. (0,1) .. controls (0,1.4) and (.4,1.65) .. (.5,1.7);
\draw [double, directed=0.55] (.5,.3) .. controls (.6,.35) and (1,.6) .. (1,1) .. controls (1,1.4) and (.6,1.65) .. (.5,1.7);
\draw [very thick, ->] (.5,1.7) -- (.5,2);
\end{tikzpicture}
\quad= \quad
\begin{tikzpicture}[anchorbase, scale=.5]
\draw [very thick,->] (.5,0) -- (.5,2);
\end{tikzpicture}
\quad= \quad
\begin{tikzpicture}[anchorbase, scale=.5]
\draw [very thick] (.5,0) -- (.5,.3);
\draw [double, directed=0.55] (.5,.3) .. controls (.4,.35) and (0,.6) .. (0,1) .. controls (0,1.4) and (.4,1.65) .. (.5,1.7);
\draw [very thick] (.5,.3) .. controls (.6,.35) and (1,.6) .. (1,1) .. controls (1,1.4) and (.6,1.65) .. (.5,1.7);
\draw [very thick, ->] (.5,1.7) -- (.5,2);
\end{tikzpicture}\\
\begin{tikzpicture}[fill opacity=.2,anchorbase,scale=.3]
\draw[very thick, directed=.55] (1,0) to [out=0,in=270] (2,1) to [out=90,in=0] (1,2)to [out=180,in=90] (0,1)to [out=270,in=180] (1,0);
\end{tikzpicture} 
\quad=\quad
$(t+ t^{-1}) \emptyset$
\quad=\quad 
\begin{tikzpicture}[fill opacity=.2,anchorbase,scale=.3]
\draw[very thick, rdirected=.55] (1,0) to [out=0,in=270] (2,1) to [out=90,in=0] (1,2)to [out=180,in=90] (0,1)to [out=270,in=180] (1,0);
\end{tikzpicture}
\end{center}
we see that
\begin{center}
\begin{tikzpicture}[anchorbase, scale=.4]
\draw [very thick, <-] (3,-1) to [out=180,in=-90] (0,1);
\draw [white,line width=.15cm] (2.05,0) to [out=180,in=0] (.05,-1) ;
\draw [very thick] (2,1) to [out=-90,in=0] (-1,-1);
\draw [very thick] (1,2) to [out=180,in=90] (0,1);
\draw [very thick] (1,2) to [out=0,in=90] (2,1);
\end{tikzpicture}    
\;\;=\;\;
$t$   \;
\begin{tikzpicture}[anchorbase, scale=.4]
\draw [very thick, <-] (2,1) to (0,1);
\end{tikzpicture}
\end{center}

\noindent
Hence we see that each twist contributes a power of $t$ and what is left is the longitudinal element $(1)$.  In the case that $k$ is negative we would have instead have negative power of $t$ for each twist. And hence in general we have $x_{1,k}=t^k(1)$.

This is done in a similar fashion for $x_{-1,k}$, except it is now the the curve now has reverse orientation.
\end{proof}

The following lemma allows us to recursively define $x_{m,n}$ using the base elements $x_{0,k}$, $x_{1,k}$, and $x_{-1,k}$, which we have already explicitly computed. 
\begin{lemma}
\label{recx}
For $m\geq 2$ we have
\begin{align*}
x_{m+1,k}&=(1,0)_T\cdot x_{m,k}-\wedge^{(1,0)}\cdot x_{m-1,k},\\
x_{-m-1,k}&=(-1,0)_T\cdot x_{-m,k}-\wedge^{(-1,0)}\cdot x_{-m+1,k}.
\end{align*}
\end{lemma}
\begin{proof}
Letting $m\geq 2$ and using the Frohman-Gelca formula (\ref{FGform}) on $(m+1,k)_T$, we have
\begin{align*}
(m+1,k)_T&=(1,0)_T*(m,k)_T-(1-m,0-k)_T*\wedge^{(m,k)}\\
(\ref{switch})&=(1,0)_T*(m,k)_T-t^{2((1-m)k-(-k)m))}\wedge^{(m,k)}*(1-m,0-k)_T\\
&=(1,0)_T*(m,k)_T-t^{2k}\wedge^{(m,k)}*(1-m,0-k)_T\\
(\ref{wedgecombinesplit})&=(1,0)_T*(m,k)_T-t^{2k}[t^{-2(k)}\wedge^{(1,0)}*\wedge^{(m-1,k)}]*(-(m-1),0-k)_T\\
(\ref{revorien})&=(1,0)_T*(m,k)_T-\wedge^{(1,0)}*(m-1,k)_T
\end{align*}
Applying the projection map $\pi:G_T(\partial M)\rightarrow G_t(M)$ to this equation we see that
\begin{align*}
\pi((m+1,k)_T)=\pi((1,0)_T*(m,k)_T-\wedge^{(1,0)}*(m-1,k)_T)   
\end{align*}
which then by the projection formula \ref{projectionaction}, becomes
$$x_{m+1,k}=(1,0)_T\cdot x_{m,k}-\wedge^{(1,0)}\cdot x_{m-1,k}.$$
Similarly we also see that
\begin{align*}
(-(m+1),k)_T&=(-1,0)_T*(-m,k)_T-(-1+m,0-k)_T*\wedge^{(-m,k)}\\
(\ref{switch})&=(-1,0)_T*(-m,k)_T-t^{2((-1+m)k-(-k)(-m)))}\wedge^{(-m,k)}*(-1+m,0-k)_T\\
&=(-1,0)_T*(-m,k)_T-t^{-2k}\wedge^{(-m,k)}*(-1+m,0-k)_T\\
(\ref{wedgecombinesplit})&=(-1,0)_T*(-m,k)_T-t^{-2k}[t^{2(k)}\wedge^{(-1,0)}*\wedge^{(-(m-1),k)}]*(m-1,0-k)_T\\
(\ref{revorien})&=(-1,0)_T*(-m,k)_T-\wedge^{(-1,0)}*(-m+1,k).
\end{align*}
Applying the projection map formula \ref{projectionaction} to this equation, we see that
$$x_{-m-1,k}=(-1,0)_T\cdot x_{-m,k}-\wedge^{(-1,0)}\cdot x_{-m+1,k}.$$
\end{proof}

\subsection{Proof Theorem \ref{actthm}, The Action}
Before we can prove Theorem \ref{actthm} we introduce one lemma that will be used often.

\begin{lemma}
Let $W\in G_t(\BT)$ such that it has no meridional component, namely it is a finite linear combination of elements of the form $(m,0)*\wedge^{(r,0)}$ where $m,r\in\BZ$. Then for all $U\in G_t(S^1\times \BD)$ we have
\begin{align}
\label{longprojez}
W\cdot U= \pi(W)*U.    
\end{align}
\end{lemma}
\begin{proof}
Note that $W\cdot U=\pi(W)\sqcup U$ by definition.
Now note that we can express $U$ in a way such that it has no meridional components, since it is an element of $G_T(S_1\times \BD)$ which has a basis consisting of purely longitudinal elements.

Now since $W$ also does not have meridional components, this means we can isotope $\pi(W)$ and $U$ in the $S^1\times \BD$ such that it has no crossings. Furthermore, after identifying $S^1\times \BD$ with the annulus we can isotope $\pi(W)$ and $U$ to be in line with the definition of $\pi(W)*U$. Namely, stacking $\pi(W)$ in the annulus atop of $U$ in the annulus. Hence we have $\pi(W)\sqcup U= \pi(W)*U$.
\end{proof}

We now prove the formulas
\begin{align*}
(m,n)_T\cdot\left[(k)_T*\wedge^{(l)}\right]&=t^{-2nl} \wedge^{(l)}* x_{m+k,n}+t^{-2n(k+l)}\wedge^{(k+l)}* x_{m-k,n}\\
\wedge^{(m,n)}\cdot\left[(k)_T*\wedge^{(l)}\right]&=t^{-2n(m+k+2l)}(k)_T*\wedge^{(l+m)},
\end{align*}
as stated in Theorem \ref{actthm}.
\begin{proof}
Note that
\begin{align*}
(m,n)_T*(k,0)_T&=(m+k,n)_T+(m-k,n)_T*\wedge^{(k,0)}
\end{align*}
and thus we have
\begin{align*}
(m,n)_T* [(k,0)_T*\wedge^{(l,0)}]&=[(m,n)_T*(k,0)_T]*\wedge^{(l,0)}\\
&=[(m+k,n)_T+(m-k,n)_T*\wedge^{(k,0)}]*\wedge^{(l,0)}\\
&=(m+k,n)_T*\wedge^{(l,0)}+(m-k,n)_T*\wedge^{(k+l,0)}\\
(\ref{switch})&=t^{2(-nl)}\wedge^{(l,0)}*(m+k,n)_T+t^{2(-n(k+l))}\wedge^{(k+l,0)}*(m-k,n)_T.
\end{align*}
Applying the projection map to this equation gives us
$$\pi((m,n)_T*[(k,0)_T*\wedge^{(l,0)}])=\pi (t^{2(-nl)}\wedge^{(l,0)}*(m+k,n)_T+t^{2(-n(k+l))}\wedge^{(k+l,0)}*(m-k,n)_T)$$
Note that the right hand side becomes
\begin{align*}
& \pi(t^{2(-nl)}\wedge^{(l,0)}*(m+k,n)_T+t^{2(-n(k+l))}\wedge^{(k+l,0)}*(m-k,n)_T) \\
(\ref{projectionaction})&=t^{-2nl} \wedge^{(l,0)}\cdot x_{m+k,n}+t^{-2n(k+l)}\wedge^{(k+l,0)}\cdot x_{m-k,n}\\
(\ref{longprojez})&=t^{-2nl} \wedge^{(l)}*x_{m+k,n}+t^{-2n(k+l)}\wedge^{(k+l)}*x_{m-k,n}.
\end{align*}
Note that using the projection formula \ref{projectionaction}, the left hand side becomes
\begin{align*}
\pi((m,n)_T*[(k,0)_T*\wedge^{(l,0)}])&=(m,n)_T\cdot [\pi((k,0)_T*\wedge^{(l,0)})]\\
(\ref{projectionaction})&=(m,n)_T\cdot [(k,0)_T\cdot\pi(\wedge^{(l,0)})]\\
(\ref{longprojez})&=(m,n)_T\cdot [(k,0)_T\cdot\wedge^{(l)}]\\
&=(m,n)_T\cdot [(k)_T*\wedge^{(l)}].
\end{align*}
Hence we then have
$$(m,n)_T\cdot\left[(k)_T*\wedge^{(l)}\right]=t^{-2nl} \wedge^{(l)}* x_{m+k,n}+t^{-2n(k+l)}\wedge^{(k+l)}* x_{m-k,n}.$$\\

Now for the 2-labelled action. Note that we have
\begin{align*}
\wedge^{(m,n)}* [(k,0)_T*\wedge^{(l,0)}]&=[t^{-2(mn)}\wedge^{(m,0)}*\wedge^{(0,n)}]* [(k,0)_T*\wedge^{(l,0)}]\\
&=t^{-2(mn)}\wedge^{(m,0)}*[\wedge^{(0,n)}* (k,0)_T]*\wedge^{(l,0)}\\
(\ref{switch})&=t^{-2(mn)}\wedge^{(m,0)}*[t^{-2(kn)}(k,0)_T*\wedge^{(0,n)}]*\wedge^{(l,0)}\\
&=t^{-2(mn+kn)}\wedge^{(m,0)}*(k,0)_T*[\wedge^{(0,n)}*\wedge^{(l,0)]}]\\
(\ref{wedgecombinesplit})&=t^{-2(mn+kn)}\wedge^{(m,0)}*(k,0)_T*[t^{2(-nl)}\wedge^{(l,n)}]\\
&=t^{-2(mn+kn+nl)}\wedge^{(m,0)}*(k,0)_T*\wedge^{(l,n)}.
\end{align*}
With this, applying the projection map we have
$$\pi(\wedge^{(m,n)}\cdot[(k,0)_T*\wedge^{(l,0)}])=\pi(t^{-2(mn-nk+nl)}\wedge^{(m,0)}*(k,0)_T*\wedge^{(l,n)}).$$
Note that the right hand side becomes
\begin{align*}
&\pi(t^{-2(mn-nk+nl)}\wedge^{(m,0)}*(k,0)_T*\wedge^{(l,n)})\\
(\ref{projectionaction})&=[t^{-2(mn+kn+nl)}\wedge^{(m,0)}*(k,0)_T]\cdot \pi(\wedge^{(l,n)})\\
(\ref{projwedge})&=[t^{-2(mn+kn+nl)}\wedge^{(m,0)}*(k,0)_T]\cdot t^{-2ln}\wedge^{(l)}\\
&=t^{-2(mn+kn+nl)}\wedge^{(m,0)}\cdot[(k,0)_T\cdot t^{-2ln}\wedge^{(l)}]\\
(\ref{longprojez})&=t^{-2(mn+kn+nl)}\wedge^{(m,0)}\cdot[(k)_T * t^{-2ln}\wedge^{(l)}]\\
(\ref{longprojez})&=t^{-2(mn+kn+nl)}\wedge^{(m)}*(k)_T * t^{-2ln}\wedge^{(l)}\\
(\ref{solidezrelations})&=t^{-2(mn+kn+2nl)}(k)_T*\wedge^{(l+m)}\\
&=t^{-2n(m+k+2l)}(k)_T*\wedge^{(l+m)}.
\end{align*}
And that the left hand side becomes
\begin{align*}
\pi(\wedge^{(m,n)}\cdot[(k,0)_T*\wedge^{(l,0)}])=\wedge^{(m.n)}\cdot[(k)_T*\wedge^{(l)}].
\end{align*}
And hence we finnally have that
$$\wedge^{(m,n)}\cdot\left[(k)_T*\wedge^{(l)}\right]=t^{-2n(m+k+2l)}(k)_T*\wedge^{(l+m)}.$$
\end{proof}

\section{Lens Spaces}
\label{secLens}
\subsection{Handle Bodies and Skein Modules}
We recall definitions regarding handles in dimension 3. Here we will be following \cite{GS99,Sch14}.
\begin{definition}
Let $M$ be a $3$-manifold with boundary $\partial M$. In this paper we will be using handles in the context of $3$-dimensions. For convenience we will use "$k$-handle" to mean a 3-dimensional $k$-handle.
\begin{itemize}
\item
 
A \textbf{1-handle} is the 3-manifold $ [-1,1]\times \BD$. One can attach a 1-handle to $M$ via an embedding $f:\partial([-1,1])\times \BD \rightarrow \partial M$. We call this a \textbf{1-handle attached to $M$}.
The \textbf{core of a 1-handle} is $[-1,1]\times {0}$ and
the \textbf{cocore of a 1-handle} is ${0}\times \BD$. 

\item
A \textbf{2-handle} is the 3-manifold $\BD\times [-1,1]$. One can attach a $2$-handle to $M$ via an embedding $\phi: \partial (\BD)\times [-1,1]\rightarrow \partial M$. We call this a \textbf{2-handle attached to $M$}.
The \textbf{core of a 2-handle} is $\BD\times \{0\}$ and
the \textbf{cocore of a 2-handle} is $\{0\}\times [-1,1]$. 

Let $N$ be a 3-dimensional manifold obtained attaching a 2-handle to a 3-manifold $M$. An isotopy that moves an edge of a web $w$ across the core of the $2$-handle is called a \textbf{handle slide}. 
    
\item 
A \textbf{3-handle} is the 3-dimensional closed ball $B$. One can attach a 3-handle to $M$ via an embedding (if it exists) $\phi:\partial (\mathbb{B})\rightarrow \partial M$. We call this a \textbf{3-handle attached to $M$}.
The \textbf{core of a  3-handle} is $\mathbb{B}$ and
the \textbf{cocore of a  3-handle} is $\{0\}$. 
\end{itemize}

A \textbf{3-dimensional $g$-handlebody} is the 3-manifold constructed by successively attaching $g$ 1-handles to the closed 3-dimensional ball. This 3-manifold is diffeomorphic to a closed regular neighborhood of a wedge of $g$ circles in $\BR^3$
\end{definition}
We now go over properties about $\mathfrak{gl_2}$-skein modules in regards to attaching handles. 
The following lemmas are from \cite{Prz99}, and we also follow \cite{Mcl04}. 
\begin{lemma}\cite{Prz99}
\label{Skein3ball}
Let $M$ be a 3-manifold with boundary $\partial M$. If $N$ is obtained from $M$ by attaching a $3$-handle $H_3$, and $i:M\rightarrow N$ the associated embedding, then the induced map $i_*:G_t(M)\rightarrow G_t(N)$ is an isomorphism.
\end{lemma}






\begin{lemma}\cite{Prz99}
\label{Sk2Hand}
Let $\gamma$ be a simple closed curve on $\partial M$ and $N=M\cup_\gamma H_2 $ be the 3-manifold obtained from $M$ by attaching a $2$-handle via the attaching map induced by $\gamma$. 
\begin{enumerate}
    \item $i_*:G_t(M)\rightarrow G_t(N)$ is surjective
    
    \item $G_t(N)\cong G_t(M)/J$ where $J$ is the submodule of $G_t(M)$ generated by expressions $w-sl(w)$, where $w$ is a tangled $\mathfrak{gl}_2$-web and $sl(w)$ is a tangled $\mathfrak{gl}_2$-web obtained from performing a handle slide with $w$ in regards to $H_2$.

\end{enumerate}

\end{lemma}





It should be noted that the proofs of these lemmas given in \cite{Prz99} are topological. Namely, the key arguments are about isotopies of links, and hence these arguments extend to the case of webs.

\subsection{Heegaard Splittings and Heegaard Diagrams}
We recall Heegaard splittings and diagrams, following \cite{Hom20, Sch14} as reference.

\begin{definition}
Let $Y$ be a 3-manifold. A \textbf{Heegaard splitting} of $Y$ is a decomposition $Y=V\cup_f W$ where $V$, $W$ are handlebodies and $f$ is an orientation reversing homeomorphism from $\partial V$ to $\partial W$.
\end{definition}

\begin{definition}
Let $H$ be a handlebody of genus $g$. A \textbf{set of attaching circles} for $H$ is a set $\{\gamma_1,\dots,\gamma_g\}$ of simple closed curves in $\Sigma=\partial H$ such that
\begin{enumerate}
    \item the curves are pairwise disjoint
    \item $\Sigma-\gamma_1-\cdots-\gamma_g$ is connected,
    \item each $\gamma_i$ bounds a disk in $H$.
\end{enumerate}
\end{definition}

\begin{definition}
    
    

\label{HeegDiag}
A \textbf{Heegaard diagram} is a triple $\mathcal{H}=(\Sigma,\alpha,\beta)$ where
\begin{enumerate}
    \item $\Sigma$ is closed oriented surface of genus $g$,
    
    \item $\alpha=\{\alpha_1,\dots,\alpha_g\}$ is a set of attaching circles for a handle body $V$ with $\partial V=\Sigma$,
    
    \item $\beta=\{\beta_1,\dots \beta_g\}$ is a set of attaching circles for a handle body $W$ with $\partial W=\Sigma$.
\end{enumerate}

Given a Heegaard diagram, one constructs the manifold it represents in the following way: Take the surface $\Sigma$ and look at the 3-manifold $\Sigma\times [0,1]$. Note that $\partial(\Sigma\times [0,1])=\Sigma\times \{0\}\sqcup \Sigma\times \{1\}$. Now attach $2$-handles along $\alpha_i\times \{0\}$ followed by attaching a $3$-handle to the boundary sphere created by the $2$-handles attached to the set of attaching circles $\alpha_i$. This process creates the handle body $V$. Now we attach 2-handles along the curves $\beta_i\times \{1\}$ followed by attaching a 3 handle to boundary sphere created by attaching 2-handles to the set of attaching circles $\beta_i$. This process creates $W$ glued to $V$ and results in a $3$-manifold.

\end{definition}

With this we would like to look at Lemma \ref{Sk2Hand} in the case of the 2-handle attachments described above.
\begin{lemma}\label{Heeg2handSkein} \cite{Mcl04, Prz99}
Let $(\Sigma,\alpha,\beta)$ be Heegaard diagram, and $M$ be the constructed 3-manifold as described in Definition \ref{HeegDiag}. Then we have
$$G_t(M)=G_t(\Sigma)/(S_V+S_W)$$
where
\begin{align*}
S_V &= \sum_{i=1}^g J_{\alpha_i}\quad, &
S_W &= \sum_{i=1}^g J_{\beta_i}
\end{align*}
and $J_{\alpha_i},J_{\beta_i}$ are the ideals generated by handle slides in regards to the 2-handles $\alpha_i,\beta_i$ as described in Lemma \ref{Sk2Hand}.
\end{lemma}


\subsection{Heegaard Splittings and Skein Modules}
\label{SkeinHeegSec}
Let $V$ and $W$ be handlebodies of genus $g$ with $\partial W=\partial V =\Sigma$, and $\tilde{f}:\partial V\rightarrow \partial W$ be an orientation reversing homeomorphism from the boundary of $V$ to the boundary of $W$. By gluing $V$ to $W$ along $\tilde{f}$ we obtain the compact oriented $3$-manifold
$$M=V\cup_{\tilde{f}} W.$$
Note that the map $\tilde{f}$ induces a map $f:G_t(\partial V)\rightarrow G_t(\partial W)$. This map is actually an anti homomorphism, namely we have
$$f(A*_{\partial V}B)=f(B)*_{\partial W}f(A),$$
where $*_{\partial V}$ and $*_{\partial W}$ denotes the multiplication in $G_t(\partial V)$ and  $G_t(\partial W)$ respectively.
With this we can view $G_t(V)$ as a left $G_t(\partial V)$-module and $G_t(W)$ as right $G_t(\partial V)$-module. Where the action can be seen as follows,
\begin{align*}
A\cdot v&:=\pi_V(A)\sqcup v\\\\  
A\cdot' w&:=\pi_W(f(A))\sqcup w.
\end{align*}
Here $A\in G_t(\partial V)$, $\pi_V$ the the projection map induced from $i_V:\partial V\rightarrow V$,  and $\pi_W$ denotes the projection map induced from $i_W:\partial W\rightarrow W$.

We see that $G_t(V)$ and $G_t(W)$ are left and right $G_t(\partial V)$-modules respectively. Namely, we see that for $A,B\in G_t(\partial V)$
\begin{align*}
A\cdot[B\cdot v] &=A\cdot[\pi_V(B)\sqcup v]\\
&=\pi_V([A*_{\partial V} B])\sqcup v\\
&=[A*_{\partial V}B]\cdot v\\\\
B\cdot'[A\cdot'w]&=B\cdot'[\pi_W(f(A))\sqcup w]\\
&=\pi_W(f(B)*_{\partial W} f(A))\sqcup w\\
&=\pi_W(f(A*_{\partial V}B))\sqcup w\\
&=[A*_{\partial V}B] \cdot' w\\
&=[B*^{op}_{\partial V} A]\cdot'w,
\end{align*}
and that the empty skein in $G_t(\partial V)$ acts as the identity on skeins in both $G_t(V)$ and $G_t(W)$.\\

One can also express the right action of $G_t(\partial V)$ on $G_t(W)$ via the canonical left action of $G_t(\partial W)$ on $G_t(W)$. Namely we see that,
\begin{align}
\label{righttoleftaction}
A\cdot'w= f(A)\cdot w
\end{align}
and furthermore as it will used later on, we see that
\begin{align}
\label{rightleftantihom}
[A*_{\partial V}B]\cdot' w &= [f(B)*_{\partial W}f(A)]\cdot w     
\end{align}
where $\cdot w $ denotes the canonical left action of $G_t(\partial W)$ on $G_t(W)$.\\

Now we want to show that when $M=V\cup_{\tilde{f}}W$ we have $G_t(M)\cong G_t(W)\otimes_{G_t(\Sigma)}G_t(V)$. To do this we first need the following lemma.

\begin{lemma}\label{Bourbaki} \cite{Bou89}
Let $A$ be a $\BC(t)$-algebra, $E'$ be a submodule of a right $A$-module $E$, and $F'$ a submodule of a left $A$-module $F$. So there are the sub-$\BC(t)$-modules $Im(E'\otimes_A F)$ and $Im(E\otimes_A F')$ of $E\otimes_A F$ given by the respective images of the canonical mappings
\begin{align*}
E'\otimes_A F &\rightarrow E\otimes_A F, \\
E\otimes_A F' &\rightarrow E\otimes_A F .
\end{align*}
Then there is a canonical $\BC(t)$-module isomorphism
$$\pi: (E/E')\otimes_A (F/F')\rightarrow (E\otimes_A F)/(Im(E'\otimes_A F)+Im(E\otimes_A F'))$$
such that, $e\in E/E'$, $f\in F/F'$, $\pi(e\otimes f)$ is the class of all elements $x\otimes y\in E\otimes_A F$ such that $x\in e$ and $y\in f$.
\end{lemma}

\begin{theorem}\cite{Mcl04} Let $M$ be a 3-manifold with Heegaard splitting $M=V\cup_{\tilde{f}} W$ with $\partial V=\partial W = \Sigma$. Then we have that
$$G_t(M)=G_t(W)\bigotimes_{G_t(\Sigma)}G_t(V).$$
\end{theorem}
\begin{proof}
Let $(\Sigma,\alpha,\beta)$ be a Heegaard diagram describing $M=V\cup_{\tilde{f}} W$. Note that $V$ can be obtained from $\Sigma$ by attaching a $2$-handle along $\alpha_i$ followed by attaching a 3-handle. From Lemma \ref{Sk2Hand} and Lemma \ref{Skein3ball} we have that $G_t(V)\cong G_t(\Sigma)/S_V$ where $S_V$ is the ideal generated by the difference of webs and handle slides of those webs across $\alpha_i$, as defined in lemmas \ref{Sk2Hand} and \ref{Heeg2handSkein}.\\

\noindent
Similarly $W$ can be obtained from $\Sigma$ by attaching a $2$-handle along $\beta_i$ followed by attaching a $3$-handle. From Lemma \ref{Sk2Hand} and Lemma \ref{Skein3ball} we have that $G_t(W)\cong G_t(\Sigma)/S_W$ where $S_W$ is the ideal generated by handle slides of webs across $\beta_i$.\\

\noindent
We will now use Lemma \ref{Bourbaki}. Let $E=G_t(\Sigma)$, $F=G_t(\Sigma)$, and $A=G_t(\Sigma)$. Here the left action of $A$ on $E$ is given by $a\cdot e= a*e$ and the right action of $A$ on $F$ is given by $a\cdot' f= f*a$ where $a\in A, e\in E, f\in F$. Now let $E'=S_W$ and $F'=S_V$, we see that $E/E'\cong G_t(V)$ and $F/F'\cong G_t(W)$.\\

\noindent
Now note that that $S_W\otimes_{G_t(\Sigma)}G_t(\Sigma)\cong S_W$ since $S_W$ is a $G_t(\Sigma)$-submodule. Similarly we have $G_t(\Sigma)\otimes_{\Sigma} S_V\cong S_V$. Using Lemma \ref{Bourbaki} we have
$$G_t(\Sigma)/S_W\otimes_{G_t(\Sigma)}G_t(\Sigma)/S_V\cong \left[G_t(\Sigma)\otimes_{G_t(\Sigma)}G_t(\Sigma)\right]/(S_W+S_V).$$
We see that $G_t(\Sigma)\otimes_{G_t(\Sigma)}G_t(\Sigma)\cong G_t(\Sigma)$.
Now note from Lemma \ref{Heeg2handSkein} we have that $G_t(\Sigma)/(S_W+S_V)\cong G_t(M)$. And hence we finally have
$$G_t(W)\otimes_{G_t(\Sigma)}G_t(V)\cong G_t(M).$$
\end{proof}
It should be noted that the proofs here are identical to the proofs found in the case of the Kauffman bracket skein module.

\subsection{Lens Spaces via Tensor Product}
The following lemma will be used to describe the tensor product structure of $G_t(L(p,q))$.
\begin{lemma}
Using the same notation introduced and used in Section \ref{ProjNotation}. We have the following equations,
\begin{align}
\label{pullx}
x_{m,n}*y_{r,s}&=t^{2rn}(m,n)_T\cdot y_{r,s}\\
\label{pully}
x_{m,n}*y_{r,s}&=t^{2ms}\wedge^{(r,s)}\cdot x_{m,n}.
\end{align}
\end{lemma}
\begin{proof}
First note and recall that from equations \ref{projwedge} and \ref{0guys} we have that
\begin{align}\label{lemlem410form}
\begin{split}
y_{r,s}&=t^{-2rs}\wedge^{(r)}\\  
(0)_T&=2.
\end{split}
\end{align}
Proving Equation \ref{pullx},
\begin{align*}
(m,n)_T\cdot y_{r,s}&=(m,n)_T\cdot  \left[\frac{t^{-2rs}}{2}(0)_T*\wedge^{(r)}\right]\hspace{15mm} (\ref{lemlem410form})\\
(\ref{actform})&=\frac{t^{-2rs}}{2} \left(t^{-2nr}\wedge^{(r)}*x_{m,n}+t^{-2nr}\wedge^{(r)}*x_{m,n}\right)\\
&=t^{-2r(n+s)}\wedge^{(r)}*x_{m,n}\\
&=t^{-2r(n+s)}x_{m,n}*\wedge^{(r)}\\
&=t^{-2rn}t^{-2rs}x_{m,n}*\wedge^{(r)}\\
(\ref{lemlem410form})&=t^{-2rn}x_{m,n}*y_{r,s}.
\end{align*}
Proving Equation \ref{pully},
\begin{align*}
\wedge^{(r,s)}\cdot x_{m,n}&=\left[\wedge^{(r,s)}*(m,n)_T\right]\cdot 1\hspace{15mm} (\ref{projectionaction})\\
(\ref{switch})&=\left[t^{-2(ms-nr)}(m,n)_T*\wedge^{(r,s)}\right]\cdot 1\\
&=t^{-2(ms-nr)}(m,n)_T\cdot y_{r,s}\\
(\ref{lemlem410form})&=t^{-2(ms-nr)}(m,n)_T\cdot\left[\frac{t^{-2rs}}{2}(0)_T*\wedge^{(r)}\right]\\
(\ref{actform})&=t^{-2(ms-nr)}t^{-2rn}x_{m,n}*y_{r,s}\\
&=t^{-2ms}x_{m,n}*y_{r,s}.
\end{align*}
\end{proof}
Now let us look at Heegaard splittings for $L(p,q)$. One has
$$L(p,q)=V\cup_{\tilde{f}} W$$
where $V$ and $W$ are solid tori, and the gluing map $\tilde{f}:\partial V=\BT \rightarrow \partial W =\BT$ can be described via the matrix
\begin{align}\label{gluematrix}
\begin{pmatrix}
a & p \\
b & q
\end{pmatrix}
\end{align}
As it will be used later on, note that since the gluing map is an orientation reversing homemorphism, the determinant of this matrix is $-1$. With this we have
$$G_t(L(p,q))=G_t(S^1\times \BD)\bigotimes_{G_t(\BT)} G_t(S^1\times \BD).$$
Recall that with the same set up as Subsection \ref{SkeinHeegSec}, we view the $G_t(S^1\times \BD)$ on the right as a left module over $G_t(\BT)$ using the canonical left action and the $G_t(S^1\times \BD)$ on the left as a right module over $G_t(\BT)$ using the induced map $f:G_t(\BT)\rightarrow G_t(\BT)$. Details on these actions were discussed in Subsection \ref{SkeinHeegSec}. The map $f$ is also described by using the matrix listed as Equation \ref{gluematrix}. With this, we have the following equations by using the definition of tensor product of left and right modules over an algebra:
\begin{align}\label{mainbalancetens}
\begin{split}
v\otimes (m,n)_T\cdot w &= (m,n)_T\cdot' v\otimes w \\\\
v\otimes \wedge^{(r,s)}\cdot w &= \wedge^{(r,s)}\cdot' v\otimes w,
\end{split}
\end{align}
where $v\in G_t(S^1\times \BD)$ on the left, $w\in G_t(S^1\times \BD)$ on the right, and $m,n,r,s\in \BZ$.
Using this, we also have the following equations that describe the tensor product structure:
\begin{align}
\label{balance}
\begin{split}
1\otimes x_{m,n} &= 1\otimes (m,n)_T\cdot 1  \hspace{15mm} (\ref{projectionaction})\\
(\ref{mainbalancetens})&=(m,n)_T\cdot' 1 \otimes 1\\
(\ref{righttoleftaction})&=(am+pn,bm+qn)_T\cdot 1 \otimes 1\\
(\ref{projectionaction})&=x_{am+pn,bm+qn}\otimes 1,\\\\
1\otimes y_{m,n} &= 1\otimes \wedge^{(m,n)}\cdot 1\\
(\ref{mainbalancetens})&=\wedge^{(m,n)}\cdot' 1 \otimes 1\\
(\ref{righttoleftaction})&=\wedge^{(am+pn,bm+qn)}\cdot 1 \otimes 1\\
(\ref{projectionaction})&=y_{am+pn,bm+qn}\otimes 1,\\\\
1\otimes x_{m,n}*y_{r,s} &=1\otimes t^{2rn}(m,n)_T\cdot y_{r,s}\hspace{15mm}(\ref{pullx})\\
(\ref{mainbalancetens})&=t^{2rn}(m,n)_T\cdot' 1\otimes y_{r,s}\\
(\ref{righttoleftaction})&=t^{2rn}(am+pn,bm+qn)_T\cdot 1\otimes y_{r,s}\\
(\ref{projectionaction})&=t^{2rn}x_{am+pn,bm+qn}\otimes y_{r,s}\\\\
1\otimes x_{m,n}*y_{r,s} &=1\otimes t^{2ms}\wedge^{(r,s)}\cdot x_{m,n}\hspace{15mm}(\ref{pully})\\
(\ref{mainbalancetens})&=t^{2ms}\wedge^{(r,s)}\cdot' 1\otimes x_{m,n}\\
(\ref{righttoleftaction})&=t^{2ms}\wedge^{(ar+ps,br+qs)}\cdot 1\otimes x_{m,n}\\
(\ref{projectionaction})&=t^{2ms}y_{ar+ps,br+qs}\otimes x_{m,n}.
\end{split}
\end{align}
We see that balancing things across the tensor product is just the act of sending skeins in one solid torus to the other solid torus through the gluing map of the solid tori.

It should also be noted that from equations \ref{mainbalancetens} we have that $\{(n_1)*\wedge^{(n_2)}\otimes 1\}$ for $n_1,n_2\in\BZ$ forms a spanning set for $G_t(L(p,q))$. Since for any element $v\otimes w$ one can balance $w$ to the other side of the tensor product using equations \ref{mainbalancetens}.

\section{Proof of Theorem \ref{lenthm}}
\label{secProofs}
\subsection{Overview and Statements of the Main Lemmas \ref{mainlemma} and \ref{lexpression}}
The main objective is to show that for all $n_1,n_2\in\BZ$ we have
$$(n_1)*\wedge^{(n_2)}\otimes 1\in V$$
where $V:=span\{(n)*\wedge^{(m)}\otimes 1\}$ as a $\BC(t)$-vector space, $n\in\{0,\dots,\left\lfloor{\frac{p}{2}}\right \rfloor\}$, $m\in \{-\left\lfloor{\frac{p}{2}}\right \rfloor,\dots ,\left\lfloor{\frac{p}{2}}\right \rfloor\}$ and $\left\lfloor{\frac{p}{2}}\right \rfloor$ meaning the greatest integer less than or equal to $\frac{p}{2}$.
We use two main important lemmas to do this, we state them here and then establish them in the next subsections.\\

The following lemma does the bulk of the work in showing $(n_1)*\wedge^{(n_2)}\otimes 1\in V$, as it shows that the elements $x_{am+kp,bm+kq}*y_{ar+ps,br+qs}\otimes 1$ are contained in particular vector space. The proof of it uses the fact that one can actually reduce the subscripts of $x_{am+kp,bm+kq}*y_{ar+ps,br+qs}\otimes 1$ at the cost of scalars.
\begin{lemma}\label{mainlemma}
For every $m,k,r,s\in \BZ$ one has $$x_{am+kp,bm+kq}*y_{ar+ps,br+qs}\otimes 1 \in \Tilde{V},$$
where
$$\Tilde{V}=span\left\{(i)*\wedge^{(j)}\otimes 1 \hspace{2mm}|\hspace{2mm} -\left\lfloor{\frac{p}{2}}\right \rfloor\leq i,j\leq \left\lfloor{\frac{p}{2}}\right \rfloor\right\}$$
as a $\BC(t)$-vector space.
\end{lemma}

The following lemma allows us to express $(n_1)*\wedge^{(n_2)}\otimes 1$ in terms of projection elements $x_{m,n}$ and $y_{r,s}$. And in doing so, this allows us to use Lemma \ref{mainlemma} to reduce the subscripts as a means towards the proof Theorem \ref{lenthm}.
\begin{lemma}\label{lexpression}
For any $n_1,n_2\in \BZ$ there exists $m,k,r,s\in \BZ$ such that we have
\begin{align}\label{expression}
(n_1)*\wedge^{(n_2)}\otimes 1&=x_{ma+kp,mb+kq}*y_{ra+ps,br+qs}\otimes c 
+ \sum a_{ij}(i)*\wedge^{(j)} \otimes 1
\end{align}
where $c,a_{ij}\in \BC(t)$ and
$i$ is strictly bounded by $-n_1$ and $n_1$, and $\sum a_{ij}$ is a finite sum.
\end{lemma}

In the next subsections we go over and establish the lemmas used to prove the above lemmas.\\

For convenience, we briefly summarize the proof Theorem \ref{lenthm}. One uses Lemma \ref{lexpression} to write the elements $(n_1)*\wedge^{(n_2)}\otimes 1$ in terms of $x_{ma+kp,mb+kq}*y_{ra+ps,br+qs}$. Then using Lemma \ref{mainlemma} one shows that it is contained in the wanted vector space.\\


\subsection{Lemmas Used to Prove Lemma \ref{mainlemma}}

The following lemma will be the main tool in reducing the subscripts of $x_{ma+pk,mb+qk}*y_{ra+sp,rb+sq}\otimes 1$, and is just an algebraic description of the fact $(p,q)$ curves now bound a disk due to the gluing of the solid tori. 

\begin{lemma}
For every $k\in \BZ$ and $u\in G_t(S^1\times \BD)$ one has the identities
\begin{align}\label{reducingab}
\begin{split}
(a+kp,b+kq)_T\cdot u \otimes 1 &=(a,b)_T\cdot u\otimes t^k.\\\\
\wedge^{(a+kp,b+kq)}\cdot u \otimes 1 &=\wedge^{(a,b)}\cdot u\otimes t^{-2k},
\end{split}
\end{align}
in the setting of $G_t(S^1\times \BD)\bigotimes_{G_t(\BT)} G_t(S^1\times \BD).$
\end{lemma}
\begin{proof}
\begin{align*}
(a+kp,b+kq)_T\cdot u \otimes 1 
&=(1,k)_T\cdot' u\otimes 1\hspace{15mm} (\ref{righttoleftaction},\ref{gluematrix})\\
(\ref{mainbalancetens})&=u\otimes (1,k)_T\cdot 1\\
(\ref{projectionaction})&=u\otimes x_{1,k}\\
(\ref{initialxform})&=u\otimes t^k x_{1,0}\\
(\ref{projectionaction})&=u\otimes (1,0)_T\cdot t^k\\
(\ref{mainbalancetens})&=(1,0)_T\cdot' u\otimes t^k\\
(\ref{righttoleftaction})&=(a,b)_T\cdot u\otimes t^k. 
\end{align*}
\begin{align*}
\wedge^{(a+kp,b+kq)}\cdot u \otimes 1 &= \wedge^{(1,k)}\cdot' u \otimes 1 \hspace{15mm} (\ref{righttoleftaction},\ref{gluematrix})\\ 
(\ref{mainbalancetens})&=  u \otimes \wedge^{(1,k)}\cdot 1\\
(\ref{projectionaction})&= u\otimes t^{-2k} \wedge^{(1)}\\
(\ref{projectionaction})&=u\otimes \wedge^{(1,0)}\cdot t^{-2k}\\
(\ref{mainbalancetens})&=\wedge^{(1,0)}\cdot' u\otimes  t^{-2k}\\
(\ref{righttoleftaction})&=\wedge^{(a,b)}\cdot u \otimes t^{-2k}.
\end{align*}
\end{proof}

The following lemma will give us a range in which we can lower the numbers $n,m$ in the context of $(n)$ and $\wedge^{(m)}$.
\begin{lemma}\label{RazvanNumbers}
Let $r\in \BZ$, $a,p$ be from the gluing matrix of our lens space, and $s_0$ be an integer that minimizes the value of $|ra+ps_0|$. The minimum value $|ra+ps_0|$ can take is at most $\left\lfloor{\frac{p}{2}}\right \rfloor$.
\end{lemma}
\begin{proof}
We use the Euclidean algorithm on $ra$ via $p$. Doing so we can find $s\in \BZ$ and $w\in\BZ$ such that
$$ra=ps+w$$
and $0\geq w< p$. Note that $w=ra-ps$.

If $|w|\leq \left\lfloor{\frac{p}{2}}\right \rfloor$, then we are done. If $\left\lfloor{\frac{p}{2}}\right \rfloor <|w|< p$, then let $w'=w-p$. We see that $|w'|\leq \left\lfloor{\frac{p}{2}}\right \rfloor$, and $w'=ra-p(s+1)$.

This result is also known as the absolute remainder theorem, the proof presented here is not original, and is generally known.




\end{proof}

The following lemma allows us to reduce the subscripts of $y_{ar+ps,br+qs}$ while keeping
$x_{ma+pk,mb+qk}$ unchanged in the term $x_{ma+pk,mb+qk}*y_{ra+sp,rb+sq}\otimes 1$. This is possible because we can rearrange $y_{ra+sp,rb+sq}$ in a way that takes advantage of equations \ref{reducingab}. After this one then projects it back down to the solid tori. Explicit details are found in the proof.
\begin{lemma}\label{lredy}
Let $s_0$ be the lowest integer such that $|ar+ps_0|$ is minimized. Then for any $m,n,r,s\in \BZ$, there exist a constant $\eta\in \BZ$ such that
\begin{align}
\label{redy}
x_{m,n}*y_{ar+ps,br+qs}\otimes 1=t^\eta(x_{m,n}*y_{ra+ps_0,rb+qs_0} \otimes 1)
\end{align}
\end{lemma}
\begin{proof}
\begin{align*}
x_{m,n}*y_{ar+ps,br+qs}\otimes 1&=\wedge^{(ar+ps,br+qs)}\cdot t^{\eta_1}x_{m,n} \otimes 1 \quad \quad \text{(\ref{pully})}\\
\text{(\ref{wedgecombinesplit})}&=t^{\lambda_1}[\wedge^{(a+p(s-s_0),b+q(s-s_0))}*\wedge^{((r-1)a+ps_0,(r-1)b+qs_0)}]\cdot t^{\eta_1}x_{m,n} \otimes 1\\
&=t^{\lambda_1}\wedge^{(a+p(s-s_0),b+q(s-s_0))}\cdot[\wedge^{((r-1)a+ps_0,(r-1)b+qs_0)}\cdot t^{\eta_1}x_{m,n}] \otimes 1\\
\text{(\ref{reducingab})}&=t^{\lambda_1}\wedge^{(a,b)}\cdot[\wedge^{((r-1)a+ps_0,(r-1)b+qs_0)}\cdot t^{\eta_1}x_{m,n}] \otimes t^{-2(s-s_0)}\\
&=t^{\lambda_1}[\wedge^{(a,b)}*\wedge^{((r-1)a+ps_0,(r-1)b+qs_0)}]\cdot t^{\eta_1}x_{m,n} \otimes t^{-2(s-s_0)}\\
\text{(\ref{wedgecombinesplit})}&=t^{\lambda_1}t^{\lambda_2}\wedge^{(ra+ps_0,rb+qs_0)}\cdot t^{\eta_1}x_{m,n} \otimes t^{-2(s-s_0)}\\
\text{(\ref{pully})}&=t^{\lambda_1}t^{\lambda_2}t^{\eta_1} t^{\eta_2}x_{m,n}*y_{ra+ps_0,rb+qs_0} \otimes t^{-2(s-s_0)}
\end{align*}
where
\begin{align*}
\lambda_1&=-2[(a+p(s-s_0))((r-1)b+qs_0)-((r-1)a+ps_0)(b+q(s-s_0))]\\
\lambda_2&=2[(a)((r-1)b+qs_0)-b((r-1)a+ps_0)]\\
\eta_1&=2(m)(br+qs)\\
\eta_2&=-2(m)(rb+qs_0).
\end{align*}
and thus we see that 
$$\eta= \lambda_1+\lambda_2+\eta_1+\eta_2.$$
It should be noted that for our purposes, having the explicit values of the above constants are not of importance.
\end{proof}

\begin{lemma}
We have an alternate form of the Frohman-Gelca Formula \ref{FGform},
\begin{align}
\label{FGadj}   
(m,n)_T*(r,s)_T=(m+r,n+s)_T+\wedge^{(m,n)}*(r-m,s-n)_T.
\end{align}
\end{lemma}
\begin{proof}
Manipulating the Frohman-Gelca Formula \ref{FGform}, we see that
\begin{align*}
(m+r,n+s)_T &=(m,n)_T*(r,s)_T-(m-r,n-s)_T*\wedge^{(r,s)}\\
(\ref{switch})&=(m,n)_T*(r,s)_T-t^{2((m-r)s-(n-s)r)}\wedge^{(r,s)}*(m-r,n-s)_T\\
&=(m,n)_T*(r,s)_T-t^{2(ms-nr)}\wedge^{(r,s)}*(m-r,n-s)_T\\
(\ref{wedgecombinesplit})&=(m,n)_T*(r,s)_T\\
&\hspace{6mm}-t^{2(ms-nr)}t^{-2(m(s-n)-n(r-m))}[\wedge^{(m,n)}*\wedge^{(r-m,s-n)}]*(m-r,n-s)_T\\
&=(m,n)_T*(r,s)_T\\
&\hspace{6mm}-t^{2(ms-nr)}t^{-2(ms-nr)}[\wedge^{(m,n)}*\wedge^{(r-m,s-n)}]*(m-r,n-s)_T\\
&=(m,n)_T*(r,s)_T-[\wedge^{(m,n)}*\wedge^{(r-m,s-n)}]*(m-r,n-s)_T\\
&=(m,n)_T*(r,s)_T-\wedge^{(m,n)}*[\wedge^{(r-m,s-n)}*(m-r,n-s)_T]\\
(\ref{revorien})&=(m,n)_T*(r,s)_T-\wedge^{(m,n)}*(r-m,s-n)_T.
\end{align*}
\end{proof}

\subsection{Proof of Lemma \ref{mainlemma}}
With this we can now prove Lemma \ref{mainlemma}. Which says that for every $m,k,r,s\in \BZ$, one has $$x_{am+kp,bm+kq}*y_{ar+ps,br+qs}\otimes 1 \in \Tilde{V},$$
where
$$\Tilde{V}:=\left\{(i)*\wedge^{(j)} \hspace{2mm}|\hspace{2mm} -\left\lfloor{\frac{p}{2}}\right \rfloor\leq i,j\leq \left\lfloor{\frac{p}{2}}\right \rfloor\right\}$$
as a $\BC(t)$-vector space.
\begin{proof}
We will use induction on $m$.\\

(Base case $m=0:$)
Note that we have
\begin{align*}
x_{am+kp,bm+kq}*y_{ar+ps,br+qs}\otimes 1 &=x_{kp,kq}*y_{ar+ps,br+qs}\otimes 1\\
(\ref{pullx})&=t^{\eta_1}(kp,kq)_T\cdot y_{ar+ps,br+qs}\otimes 1\\
(\ref{mainbalancetens})&=t^{\eta_1}y_{ar+ps,br+qs}\otimes(0,k)_T\cdot1\\
(\ref{m0initform})&=t^{\eta_1}y_{ar+ps,br+qs}\otimes t^k+t^{-k}\\
&=(t^k+t^{-k})t^{\eta_1}(y_{ar+ps,br+qs}\otimes 1)
\end{align*}
where 
$\eta_1=2(ar+ps)(kq)$.
With this, we now have to show that
$(y_{ar+ps,br+qs}\otimes 1)\in \tilde{V}$. Note that
\begin{align*}
y_{ar+ps,br+qs}\otimes1 &=\frac{1}{2}x_{0,0}*y_{ar+ps,br+qs}\otimes 1 \\
\text{(\ref{redy})}&=\frac{t^{\eta}}{2}x_{0,0}*y_{ar+ps_0,br+qs_0}\otimes 1 \\
&=t^{\eta}y_{ar+ps_0,br+qs_0}\otimes 1
\end{align*}
Where $\eta$ in the equation above is from Equation \ref{redy}, but is not of importance for our purposes. Now note that Since $s_0$ minimizes $|ar+ps_0|$ and by Lemma \ref{RazvanNumbers} we know that $|ar+ps_0|$ is at most $\left\lfloor{\frac{p}{2}}\right \rfloor$, we thus have $y_{ar+ps_0,br+qs_0}\otimes 1\in \tilde{V}$.\\

(Base Case $m=1:$) We see that
\begin{align*}
x_{am+kp,bm+kq}*y_{ar+ps,br+qs}\otimes1 &=x_{a+kp,b+kq}*y_{ar+ps,br+qs}\otimes 1\\
\text{(\ref{pullx})}&=t^{\eta_2}(a+kp,b+kq)_T\cdot y_{ar+ps,br+qs}\otimes 1\\
\text{(\ref{reducingab})}&=t^{\eta_2}(a,b)_T\cdot y_{ar+ps,br+qs}\otimes t^k\\
\text{(\ref{pullx})}&=t^{\eta_2}t^{\eta_3}x_{a,b}*y_{ar+ps,br+qs}\otimes t^k\\
&=t^{\eta_2}t^{\eta_3}t^k(x_{a,b}*y_{ar+ps,br+qs}\otimes 1)\\
(\text{\ref{redy}})&=t^{\eta_2}t^{\eta_3}t^kt^{\eta_4}(x_{a,b}*y_{ar+ps_0,br+qs_0}\otimes 1),
\end{align*}
where
\begin{align*}
\eta_2&=2(ar+ps)(b+kq)\\
\eta_3&=2(ar+ps)(b)
\end{align*}
and $\eta_4$ is the constant from Lemma \ref{lredy}. Note that $a,b$ are from the gluing matrix of the lens space. And one can choose to have $a\leq \left\lfloor{\frac{p}{2}}\right \rfloor$, and hence we see that  $x_{a,b}*y_{ar+ps_0,br+qs_0}\otimes 1\in \tilde{V}$.\\

(Induction Hypothesis:) Suppose that we have the theorem true for $m-1$ and $m-2$, where $m>1$. Now let $k_0$ be an integer that minimizes $am+k_0p$ while keeping the expression positive. Now note that as a pair integers we have
$$[am+kp,bm+kq]=\Big[a+(k-k_0)p,b+(k-k_0)q\Big]+\Big[a(m-1)+k_0p,b(m-1)+k_0q\Big],$$
which then tells us that
\begin{align*}
(ma+kp,mb+kq)_T&=\Big(a+(k-k_0)p,b+(k-k_0)q\Big)_T*\Big((m-1)a+k_0p,(m-1)b+k_0q\Big)_T\\
&-\wedge^{\big(a+(k-k_0)p,b+(k-k_0)q\big)}*\Big((m-2)a-(k-2k_0)p,(m-2)b-(k-2k_0)q\Big)_T.  
\end{align*}
We see that this allows us to express terms of $m$ with terms of $m-1$ and $m-2$ which are given in the context of our induction hypothesis. 
With this, we then see that
\begin{align*}
& x_{am+kp,bm+kq}*y_{ar+ps,br+qs}\otimes 1 =  t^{\eta}(am+kp,bm+kq)_T\cdot y_{ar+ps,br+qs}\otimes 1 \quad (\ref{pullx})\\
&=t^{\eta}(a+(k-k_0)p,b+(k-k_0)q)_T*((m-1)a+k_0p,(m-1)b+k_0q)_T \cdot y_{ar+ps,br+qs}\otimes 1\\
&-t^{\eta}\wedge^{(a+(k-k_0)p,b+(k-k_0)q)}*((m-2)a-(k-2k_0)p,(m-2)b-(k-2k_0)q)_T  \cdot y_{ar+ps,br+qs}\otimes 1.
\end{align*}
Where $\eta$ in the equation above is from Equation \ref{pullx}, but is not important for purposes.\\

Let us now look at the two terms separately and show that they are both in $\tilde{V}$. We first do the easier term, namely we see that
\begin{align*}
&-\wedge^{(a+(k-k_0)p,b+(k-k_0)q)}*((m-2)a-(k-2k_0)p,(m-2)b-(k-2k_0)q)_T  \cdot y_{ar+ps,br+qs}\otimes 1\\
&=-\wedge^{(a+(k-k_0)p,b+(k-k_0)q)}\cdot[((m-2)a-(k-2k_0)p,(m-2)b-(k-2k_0)q)_T  \cdot y_{ar+ps,br+qs}]\otimes 1\\
(\ref{reducingab})
&=\wedge^{(a,b)}\cdot[((m-2)a-(k-2k_0)p,(m-2)b-(k-2k_0)q)_T]  \cdot y_{ar+ps,br+qs}\otimes t^{-2(k-k_0)}\\
&=[\wedge^{(a,b)}*((m-2)a-(k-2k_0)p,(m-2)b-(k-2k_0)q)_T]  \cdot y_{ar+ps,br+qs}\otimes t^{-2(k-k_0)}\\
(\ref{reducingab})
&=[\wedge^{(a,b)}*((m-2)a-(k-2k_0)p,(m-2)b-(k-2k_0)q)_T]  \cdot [\wedge^{(ar+ps,br+qs)}\cdot 1] \otimes t^{-2(k-k_0)}\\
&=\wedge^{(a,b)}*[((m-2)a-(k-2k_0)p,(m-2)b-(k-2k_0)q)_T * \wedge^{(ar+ps,br+qs)}]\cdot 1 \otimes t^{-2(k-k_0)}\\
(\ref{switch})
&=\wedge^{(a,b)}*[t^{\eta_5}\wedge^{(ar+ps,br+qs)}*((m-2)a-(k-2k_0)p,(m-2)b-(k-2k_0)q)_T  ]\cdot 1 \otimes t^{-2(k-k_0)}\\
(\ref{wedgecombinesplit})
&=[t^{\eta_5}t^{\eta_6}\wedge^{(a(r+1)+ps,b(r+1)+qs)}]\cdot x_{(m-2)a-(k-2k_0)p,(m-2)b-(k-2k_0)q}  \otimes t^{-2(k-k_0)}\\
(\ref{pully})
&=t^{\eta_5}t^{\eta_6}t^{\eta_7}y_{(a(r+1)+ps,b(r+1)+qs)}* x_{(m-2)a-(k-2k_0)p,(m-2)b-(k-2k_0)q}  \otimes t^{-2(k-k_0)}\\
(\ref{redy})&=t^{-2(k-k_0)}t^{\eta_5}t^{\eta_6}t^{\eta_7} t^{\eta_8}( x_{(m-2)a-(k-2k_0)p,(m-2)b-(k-2k_0)q}*y_{(a(r+1)+ps_1,b(r+1)+qs_1)}  \otimes 1)\in \tilde{V}
\end{align*}
where
\begin{align*}
\eta_5&=2(((m-2)a-(k-2k_0)p)(br+qs)-((m-2)b-(k-2k_0)q)(ar+ps))\\  
\eta_6&=2((a)(br+qs)-(b)(ar+ps))\\    
\eta_7&=-2((m-2)a-(k-k_0)p)(b(r+1)+qs),
\end{align*}
and $\eta_8$ is the constant from Lemma $\ref{lredy}$.

Now looking at the first term, we see that
\begin{align*}
&(a+(k-k_0)p,b+(k-k_0)q)_T*((m-1)a+k_0p,(m-1)b+k_0q)_T \cdot y_{ar+ps,br+qs}\otimes 1\\
&=(a+(k-k_0)p,b+(k-k_0)q)_T\cdot[((m-1)a+k_0p,(m-1)b+k_0q)_T \cdot y_{ar+ps,br+qs}]\otimes 1\\
\text{(\ref{reducingab})}&=(a,b)_T\cdot[((m-1)a+k_0p,(m-1)b+k_0q)_T \cdot y_{ar+ps,br+qs}]\otimes t^{k-k_0}\\
&=[(a,b)_T*((m-1)a+k_0p,(m-1)b+k_0q)_T]\cdot y_{ar+ps,br+qs}\otimes t^{k-k_0}\\
\text{(\ref{FGadj})}&=\left[(ma+k_0p,mb+k_0q)_T+\wedge^{(a,b)}*((m-2)a+k_0p,(m-2)b+k_0q)_T\right]\cdot y_{ar+ps,br+qs}\otimes t^{k-k_0}\\
&=(ma+k_0p,mb+k_0q)_T\cdot y_{ar+ps,br+qs}\otimes t^{k-k_0}\\
&\hspace{4mm}+\wedge^{(a,b)}*((m-2)a+k_0p,(m-2)b+k_0q)_T\cdot y_{ar+ps,br+qs}\otimes t^{k-k_0}\\
\text{(\ref{pullx})}&=t^{\gamma_1}x_{ma+k_0p,mb+k_0q}*y_{ar+ps,br+qs}\otimes t^{k-k_0}\\
&\hspace{4mm}+[\wedge^{(a,b)}*((m-2)a+k_0p,(m-2)b+k_0q)_T]* \wedge^{(ar+ps,br+qs)}\cdot 1\otimes t^{k-k_0}\\
\text{(\ref{switch})}&=t^{\gamma_1}x_{ma+k_0p,mb+k_0q}*y_{ar+ps,br+qs}\otimes t^{k-k_0}\\
&\hspace{4mm}+\wedge^{(a,b)}*[t^{\gamma_2}\wedge^{(ar+ps,br+qs)}*((m-2)a+k_0p,(m-2)b+k_0q)_T] \cdot 1\otimes t^{k-k_0}\\
&=t^{\gamma_1}x_{ma+k_0p,mb+k_0q}*y_{ar+ps,br+qs}\otimes t^{k-k_0}\\
&\hspace{4mm}+[\wedge^{(a,b)}*t^{\gamma_2}\wedge^{(ar+ps,br+qs)}]\cdot x_{(m-2)a+k_0p,(m-2)b+k_0q} \otimes t^{k-k_0}\\
\text{(\ref{wedgecombinesplit})}&=t^{\gamma_1}x_{ma+k_0p,mb+k_0q}*y_{ar+ps,br+qs}\otimes t^{k-k_0}\\
&\hspace{4mm}+t^{\gamma_2}t^{\gamma_3}\wedge^{(a(r+1)+ps,b(r+1)+qs)}\cdot x_{(m-2)a+k_0p,(m-2)b+k_0q} \otimes t^{k-k_0}\\
\text{(\ref{pully})}&=t^{\gamma_1}x_{ma+k_0p,mb+k_0q}*y_{ar+ps,br+qs}\otimes t^{k-k_0}\\
&\hspace{4mm}+t^{\gamma_2}t^{\gamma_3}t^{\gamma_4}y_{a(r+1)+ps,b(r+1)+qs}*x_{(m-2)a+k_0p,(m-2)b+k_0q} \otimes t^{k-k_0}\\
\text{(\ref{redy})}&=t^{\gamma_1}t^{\gamma_{5}}x_{ma+k_0p,mb+k_0q}*y_{ar+ps_0,br+qs_0}\otimes t^{k-k_0}\\
&\hspace{4mm}+t^{\gamma_2}t^{\gamma_3}t^{\gamma_4}t^{\gamma_{6}}x_{(m-2)a+k_0p,(m-2)b+k_0q}*y_{a(r+1)+ps_1,b(r+1)+qs_1} \otimes t^{k-k_0}\\
&=t^{\gamma_1}t^{\gamma_{5}}t^{k-k_0}(x_{ma+k_0p,mb+k_0q}*y_{ar+ps_0,br+qs_0}\otimes 1)\\
&\hspace{4mm}+t^{\gamma_2}t^{\gamma_3}t^{\gamma_4}t^{\gamma_{6}}t^{k-k_0}(x_{(m-2)a+k_0p,(m-2)b+k_0q}*y_{a(r+1)+ps_1,b(r+1)+qs_1} \otimes 1)\in \tilde{V}
\end{align*}
where
\begin{align*}
\gamma_1&=-2(ar+ps)(mb+k_0q)\\
\gamma_2&=2[((m-2)a+k_0p)(br+qs)-((m-2)b+k_0q)(ar+ps)]\\
\gamma_3&=2[(a)(br+qs)-(b)(ar+ps)]\\
\gamma_4&=-2((m-2)a+k_0p)(b(r+s)+qs)
\end{align*}
and $\gamma_5,\gamma_6$ are the respective constants from Lemma \ref{lredy}. Again it should be noted that for our purposes the actual values of the constants are not of importance.

One similarly uses induction for the case $m<-1$.
\end{proof}

\subsection{Lemmas Used to Prove Lemma \ref{lexpression}}
\begin{lemma}\label{coprime}
Every integer $n$ can be written as $ma+kp$ for some $m,k\in \BZ$. Where $a,p$ are from the gluing matrix
$$\begin{pmatrix}
a & p\\
b & q
\end{pmatrix}$$
from the Heegaard splitting of the lens space being studied.
\end{lemma}
\begin{proof}
Note 
$$
Det
\begin{pmatrix}
a & p\\
b & q
\end{pmatrix}=-1
$$
gives us that
$aq-bp=-1$ and in particular $a(-q)+p(b)=1$.
Now note the well known result that for any two integers $l$ and $j$, there exist integers $x$ and $y$ such that 
$$gcd(l,j)=l(x)+j(y).$$
This result then tells us that $gcd(a,p)=1$. This then lets us write $am+pk=n$ as a linear Diophantine equation with given integers $a,p,n$. It is a well known result that this linear Diophantine equation has solution if and only if $n$ is multiple of $gcd(a,p)$. Hence, in our case of $gcd(a,p)=1$, we have that there exists solutions $m,k\in \BZ$ so that we have $n=ma+kp.$
\end{proof}

\begin{lemma}\label{Highestpower}
For any two integers $m,n$ we have that
\begin{align*}
x_{m,n}&=\sum a_{ij}(i)*\wedge^{(j)}
\end{align*}
where $a_{ij}\in \BC(t)$, $i$ is bounded by $-m$ and $m$, and the summation over $j$ is finite.

In particular we have,
\begin{align*}
(m)&=x_{m,n}+\sum a'_{kl}(k)*\wedge^{(l)}
\end{align*}
where k is strictly bounded by $-m$ and $m$, $a'_{kl}\in\BC(t)$, and the summation over $l$ is finite.
\end{lemma}
\begin{proof}
Note that Lemma \ref{recx} allows us to view $x_{m,n}$ recursively, namely we have
\begin{align*}
x_{m+1,n}&=(1,0)_T\cdot x_{m,n}-\wedge^{(1,0)}\cdot x_{m-1,n},\\
x_{-m-1,n}&=(-1,0)_T\cdot x_{-m,n}-\wedge^{(-1,0)}\cdot x_{-m+1,n},
\end{align*}
and from Lemma \ref{initialx} and \ref{m0init}, we have the following initial conditions
\begin{align*}
x_{0,k}&=t^{k}+t^{-k}
&
x_{1,k}&=t^k(1)\\
x_{0,-k}&=t^k+t^{-k}
&
x_{-1,k}&=t^{-k}(-1).
\end{align*}

Note that 
$$(1,0)_T\cdot (k)*(l)=(k+1)*(l).$$
With this we see that the number $k$ in the context of $(k)$ in the terms of $x_{m,n}$ is bounded by $-m$ and $m$. And that the greatest numbered term in the expression $x_{m,n}$ is $(m)$. This is from the way $x_{m,n}$ is recursively defined. We also see from this recursion relation that super scripts of the $\wedge^{(j)}$ terms in the expression $x_{m,n}$ are finite.

Hence
\begin{align*}
x_{m,n}&=\sum a_{ij}(i)*\wedge^{(j)}
\end{align*}
with $i$ bounded by $m$ and $-m$, $a_{ij}\in \BC(t)$, and the summation over $j$ is finite.

Now we can isolate the $(m)$ term in the above expression, since $\wedge^{(k)}$ and $a_{ij}$ are invertible elements. And thus have
\begin{align*}
(m)&=x_{m,n}+\sum a'_{kl}(k)*\wedge^{(l)}
\end{align*}
where k is strictly bounded by $-m$ and $m$, $a'_{kl}\in\BC(t)$, and the summation over $l$ is finite.
\end{proof}

It should be noted that in the equation
$$(m)=x_{m,n}+\sum a_{kl}'(k)*(l),$$
we have that $n$ is arbitrary and that the explicit value of $a_{kl}'$ depends on $n$. For our purposes we do not need to explicitly know how $a_{kl}'$ depends on $n$.


\subsection{Proof of Lemma \ref{lexpression}}
We are now ready to prove Lemma \ref{lexpression}. Which says that for any $n_1,n_2\in \BZ$ there exists $m,k,r,s\in \BZ$ such that we have
\begin{align*}
(n_1)*\wedge^{(n_2)}\otimes 1&=x_{ma+kp,mb+kq}*y_{ra+ps,br+qs}\otimes c 
+ \sum a_{ij}(i)*\wedge^{(j)} \otimes 1
\end{align*}
where $c,a_{ij}\in \BC(t)$,
$i$ is strictly bounded by $-n_1$ and $n_1$, and $\sum a_{ij}$ is a finite sum.
\begin{proof}
First note that
$$\wedge^{(ra+sp)}\otimes 1= y_{ra+sp,br+qs}\otimes t^{-2(ra+sp)(br+qs)}.$$
With this we see that
\begin{align*}
(n_1)*\wedge^{(n_2)}\otimes 1&=(ma+kp)*\wedge^{(ra+sp)}\otimes 1 \quad \text{(Lemma \ref{coprime})}\\
\text{(Lemma \ref{Highestpower})}&=x_{ma+kp,mb+kq}*y_{ra+ps,br+qs}\otimes c + \sum a_{ij}(i)*\wedge^{(j)} \otimes 1
\end{align*}
where $c,a_{ij}\in \BC(t)$ and $|i|$ is strictly bounded by $|ma+kp|$.

It should be noted that a lot of the scalars such as $t^{-2(ra+sp)(br+qs)}$ stated at the beginning are implicitly written in $a_{ij}$.
\end{proof}

\subsection{Proof of Theorem \ref{lenthm}}
We are now ready to prove Theorem \ref{lenthm}. Which says that for the lens space $L(p,q)$, we have that $G_t(L(p,q))$ is spanned by the elements
$$\{(n)*\wedge^{(m)}\otimes 1\},$$
where $n\in\{0,\dots,\left\lfloor{\frac{p}{2}}\right \rfloor\}$ and $m\in \{-\left\lfloor{\frac{p}{2}}\right \rfloor,\dots ,\left\lfloor{\frac{p}{2}}\right \rfloor\}$ with $\left\lfloor{\frac{p}{2}}\right \rfloor$ meaning the greatest integer less than or equal to $\frac{p}{2}$.
\begin{proof}
From Lemma \ref{lexpression} we see that
\begin{align*}
(n_1)*\wedge^{(n_2)}\otimes 1&=x_{ma+kp,mb+kq}*y_{ra+ps,br+qs}\otimes c + \sum a_{ij}(i)*\wedge^{(j)} \otimes 1
\end{align*}
where $i$ is strictly bounded by $n_1$ and $-n_1$. Hence by applying Lemma \ref{mainlemma} to the term $x_{ma+kp,mb+kq}*y_{ra+ps,br+qs}$ and repeating the process on the other terms of $\sum a_{ij}(i)*\wedge^{(j)}\otimes 1$ we have that $$(n_1)*\wedge^{(n_2)}\otimes 1\in \Tilde{V}$$
for every $n_1,n_2\in \BZ$, where
$\Tilde{V}=span\{(i)*\wedge^{(j)}\otimes 1\}$ as a $\BC(t)$-vector space and $i,j\in \{-\left\lfloor{\frac{p}{2}}\right \rfloor, \dots, \left\lfloor{\frac{p}{2}}\right \rfloor\}$.\\

We now show that we can reduce the range of $i$ in the context of $(i)$ even further. Let
$$V:=span\{(n)*\wedge^{(m)}\otimes 1\}$$
as a $\BC(t)$-vector space where $n\in\{0,\dots,\left\lfloor{\frac{p}{2}}\right \rfloor\}$ and $m\in \{-\left\lfloor{\frac{p}{2}}\right \rfloor,\dots ,\left\lfloor{\frac{p}{2}}\right \rfloor\}$.

Now we will reduce the range on $n_1$ for elements $(n_1)*\wedge^{(n_2)}\otimes 1\in \Tilde{V}$ . Note for $k\in \{-\left\lfloor{\frac{p}{2}}\right \rfloor, \dots, 0\}$ we have that 
\begin{align*}
(k)*\wedge^{(n_2)}\otimes 1&= (-k)*\wedge^{(n_2+k)}\otimes 1.
\end{align*}
In the case that $n_2+k$ is not bounded by $-\left\lfloor{\frac{p}{2}}\right \rfloor$ and $\left\lfloor{\frac{p}{2}}\right \rfloor$ one can reduce the number $n_2+k$, in the context  of $\wedge^{(n_2+k)}$, using Lemma \ref{lredy}.\\

With this we now have that
$$(n_1)*\wedge^{(n_2)}\otimes 1\in V,$$
for all $n_1,n_2\in \BZ$.
\end{proof}

\end{document}